\documentclass[10pt]{article}
\usepackage{amsfonts,color}
 \usepackage{amsmath,amssymb}
 \usepackage{epsfig}
\usepackage{lscape}
\usepackage{amssymb}
\usepackage{graphicx}
\usepackage[utf8]{inputenc}	
\usepackage{amsthm,amssymb}
\usepackage{amsfonts}
\usepackage[english]{babel}
\usepackage{dsfont}
\usepackage{bbm}
\usepackage[T1]{fontenc}
\usepackage{colonequals}
\usepackage{geometry}
\usepackage{mathtools}
\usepackage{csquotes}
\usepackage{url}
\newcommand{\id}{{\boldsymbol{\mathbbm{1}}}}

\allowdisplaybreaks[1]

\makeindex

 \newtheorem{theorem}{Theorem}[section]
 \newtheorem{lemma}[theorem]{Lemma}
 \newtheorem{remark}[theorem]{Remark}
 \newtheorem{proposition}[theorem]{Proposition}
 \newtheorem{corollary}[theorem]{Corollary}

 \newcommand{\R}{\mathbb{R}}
 
 \newcommand{\K}{\mathbb{K}}

 \newcommand{\norm}[2][]{\|#2\|_{#1}}
 \newcommand{\matr}[1]{\begin{pmatrix}#1\end{pmatrix}}

\DeclareMathOperator{\diag}{diag}
\DeclareMathOperator{\Sym}{Sym}
\DeclareMathOperator{\PSym}{PSym}

\DeclareMathOperator{\dev}{dev}
\DeclareMathOperator{\GL}{GL}
\DeclareMathOperator{\SL}{SL}
\DeclareMathOperator{\SO}{SO}
\DeclareMathOperator{\OO}{O}
\DeclareMathOperator{\CSO}{CSO}

\def\barr{\begin{array}}

\DeclareMathOperator{\tr}{tr}

\def\dd{\displaystyle}

\def\barr{\begin{array}}
\def\earr{\end{array}}
\def\bec#1{\begin{equation}\label{#1}}
\def\becn{\begin{equation*}}
\def\endec{\end{equation}}
\def\endecn{\end{equation*}}

\makeatletter
\let\@fnsymbol\@arabic
\makeatother
\newcommand{\GLp}{\GL^{\!+}}
\newcommand{\GLpz}{\GLp(2)}

\newcommand{\setvert}{\:|\:}

\DeclareMathOperator{\dist}{dist}
\newcommand{\disteuc}{\dist_{\mathrm{Euclid}}}
\newcommand{\dg}{\dist_{\mathrm{geod}}}
\newcommand{\z}{z}
\newcommand{\ahalf}{{1\!/\!2}}
\newcommand{\ftilde}{{\widetilde{f}}}
\newcommand{\dx}{\mathrm{d}x}

\begin{document}
\title{Rank-one convexity implies polyconvexity for isotropic, objective and isochoric  elastic energies in the two-dimensional case}
\author{%
Robert J.\ Martin\thanks{Corresponding author: Robert J.\ Martin, Lehrstuhl f\"{u}r Nichtlineare Analysis und Modellierung, Fakult\"{a}t f\"{u}r Mathematik,
Universit\"{a}t Duisburg-Essen, Thea-Leymann Str. 9, 45127 Essen, Germany; email:   robert.martin@uni-due.de}
\quad and \quad
Ionel-Dumitrel Ghiba\thanks{Ionel-Dumitrel Ghiba, \ \  Lehrstuhl f\"{u}r Nichtlineare Analysis und Modellierung, Fakult\"{a}t f\"{u}r Mathematik,
Universit\"{a}t Duisburg-Essen, Thea-Leymann Str. 9, 45127 Essen, Germany;  Alexandru Ioan Cuza University of Ia\c si, Department of Mathematics,  Blvd.
Carol I, no. 11, 700506 Ia\c si,
Romania; and  Octav Mayer Institute of Mathematics of the
Romanian Academy, Ia\c si Branch,  700505 Ia\c si, email: dumitrel.ghiba@uni-due.de, dumitrel.ghiba@uaic.ro}
\quad and \quad
Patrizio Neff\thanks{Patrizio Neff,  \ \ Head of Lehrstuhl f\"{u}r Nichtlineare Analysis und Modellierung, Fakult\"{a}t f\"{u}r
Mathematik, Universit\"{a}t Duisburg-Essen,  Thea-Leymann Str. 9, 45127 Essen, Germany, email: patrizio.neff@uni-due.de}
}
\date{\today\vspace*{-1em}}
\newgeometry{top=2em,bottom=6.65em}
\maketitle

\begin{abstract}
We show that in the two-dimensional case, every objective, isotropic and isochoric energy function which is rank-one convex on $\GLpz$ is already polyconvex on $\GLpz$. Thus we negatively answer Morrey's conjecture in the subclass of isochoric nonlinear energies, since polyconvexity implies quasiconvexity. Our methods are based on different representation formulae for objective and isotropic functions in general as well as for isochoric functions in particular.
We also state criteria for these convexity conditions in terms of the deviatoric part of the logarithmic strain tensor.
\\[1.4em]
\textbf{Mathematics Subject Classification}: 74B20, 74G65, 26B25
\\[1.4em]
\textbf{Key words}: rank-one convexity, polyconvexity, quasiconvexity, Morrey's conjecture, isochoric energies, conformal invariance, nonlinear elasticity, volumetric-isochoric split, calculus of variations
\end{abstract}

\tableofcontents

\restoregeometry
\newgeometry{top=9em,bottom=7em,left=9em,right=9em}

\newpage

\section{Introduction}
\label{section:introduction}

We consider different convexity properties of a real-valued function $W:\GLp(2)\to\R$ on the group $\GLp(2)=\{X\in\R^{2\times 2} \,\setvert\, \det X > 0\}$ of invertible $2\times 2$-matrices with positive determinant. Our work is mainly motivated by the theory of nonlinear hyperelasticity, where $W(\nabla\varphi)$ is interpreted as the energy density of a deformation $\varphi:\Omega\to\R^2$; here, $\Omega\subset\R^2$ corresponds to a planar elastic body in its reference configuration. The elastic energy $W$ is assumed to be \emph{objective} as well as \emph{isotropic}, i.e.\ to satisfy the equality
\[
	W(Q_1\,F\,Q_2) = W(F) \quad\text{ for all }\; F\in\GLpz \;\text{ and all }\; Q_1,Q_2\in\SO(2)\,,
\]
where $\SO(2)=\{X\in \R^{2\times 2} \,|\, X^T X=\id,\;\det{X}=1\}$ denotes the special orthogonal group.

Different notions of convexity play an important role in elasticity theory. Here, we focus on the concepts of \emph{rank-one convexity}, \emph{polyconvexity} and \emph{quasiconvexity}.
Following a definition by Ball \cite[Definition 3.2]{Ball77}, we say that $W$ is  \emph{rank-one convex} on $\GLpz$ if it is convex on all closed line segments in $\GLpz$ with end points differing by a matrix of rank one, i.e
 \begin{align*}
W( F+(1-\theta)\, \xi\otimes \eta)\leq \theta \,W( F)+(1-\theta) W(F+\xi\otimes \eta)
\end{align*}
for all $F\in \GLpz$, $\theta\in[0,1]$ and all $\,\, \xi,\, \eta\in\mathbb{R}^2$ with $F+t\, \xi\otimes \eta\in \GLpz$ for all
$t\in[0,1]$, where $\xi\otimes\eta$ denotes the dyadic product.
For sufficiently regular functions $W:\GLpz\rightarrow \mathbb{R}$, rank-one convexity is equivalent to \emph{Legendre-Hadamard ellipticity} (cf.\ \cite{de2012note}) on $\GLpz$:
\begin{align*}
D^2_F W(F)(\xi\otimes\eta,\xi\otimes\eta)\geq0 \quad\text{ for all }\;\xi,\eta\in\mathbb{R}^2\setminus \{0\},\;\; F\in \GLpz\,.
\end{align*}
The rank-one convexity is connected with the study of wave propagation 
\cite{eremeyev2007constitutive,ZubovRudev,SawyersRivlin78,ChiritaGhiba1} or hyperbolicity of the dynamic problem, and plays an important role in the existence and uniqueness theory for linear elastostatics and elastodynamics \cite{Ogden83,fosdick2007note,edelstein1968note,simpson2008bifurcation}, cf.\ \cite{ernst1998ellipticity,knowles1976failure}.
It also ensures the correct spatial and temporal behaviour of the solution to the boundary value problems for a large class of materials \cite{chirictua1999time,chirictua2007strong,GhibaAn,ghiba2015spatial}. Important criteria for the rank-one convexity of functions were established by Knowles and Sternberg \cite{knowles1975ellipticity} as well as by \v{S}ilhav\'y \cite{SilhavyPRE99} and Dacorogna \cite{Dacorogna01}.
The notion of \emph{polyconvexity} was introduced into the context of nonlinear elasticity theory by John Ball \cite{Ball77,Ball78} (cf.\ \cite{Ball78,Raoult86,Dacorogna08}). In the two-dimensional case, a function $W:\GLpz\rightarrow \mathbb{R}$ is called  polyconvex if and only if it is expressible in the form
\[
	W(F) =P(F,\det F), \qquad P : \; \R^{2\times2}\times\R \;\cong\; \mathbb{R}^{5}\rightarrow\mathbb{R}\cup\{+\infty\}\,,
\]
where $P(\cdot,\cdot)$  is convex. Since the polyconvexity of an energy $W$ already implies the weak lower-semicontinuity of the corresponding energy functional, it is of fundamental importance to the direct methods in the calculus of variations. In particular, this implication is still valid for functions $W$ defined only on $\GLpz$ which do not satisfy polynomial growth conditions; this is generally the case in nonlinear elasticity.

Lastly, a function $W$ is called \emph{quasiconvex} at $\overline{F}\in \GLp(n)$ if the condition
\begin{align}
\int_{\Omega}W(\overline{F}+\nabla \vartheta)\,\dx\geq \int_{\Omega}W(\overline{F})\,\dx=W(\overline{F})\cdot |\Omega| \quad \text{for every bounded open set}
\quad \Omega\subset\mathbb{R}^n
\end{align}
holds for all $ \vartheta\in C_0^\infty (\Omega)$ such that $\det(\overline{F}+\nabla \vartheta)>0$.
Note carefully that there are alternative definitions of quasiconvexity for functions on $\GLp(n)$, cf.\ \cite{ball2015incompatible}. %
Although quasiconvexity of an energy function $W$ is sufficient for the weak lower-semicontinuity of the corresponding energy functional if $W:\R^{n\times n}\to\R$ is continuous and satisfies suitable growth conditions \cite{sivaloganathan1988implications,chelminski2006new}, it is generally not sufficient in the case of energy functions defined only on $\GLp(n)$.
\pagebreak %

It is well known that the implications
    \begin{align*}\hspace{-4mm}
     \text{polyconvexity} \quad\Longrightarrow\quad \text{quasiconvexity} \quad\Longrightarrow\quad   \text{rank-one convexity}
    \end{align*}
hold for arbitrary dimension $n$. However, it is also known that rank-one convexity does not imply polyconvexity in general (see the Alibert-Dacorogna-Marcellini example  \cite{alibert1992example}, cf.\ \cite[p.~221]{Dacorogna08} and \cite{aubert1987counterexample}), and that for $n>2$  rank-one convexity does not imply quasiconvexity \cite{Ball84b,Sverak92,Sverak98,Dacorogna08}.

The question whether rank-one convexity implies quasiconvexity in the two-dimensional case is considered to be one of the major open problems in the calculus of variations \cite{Balldoes,Ball02,conti2005rank,parry1995planar,parry2000rank}. Morrey conjectured in 1952 that the two are not equivalent \cite{morrey1952quasi,pedregal2014some,astala2012quasiconformal,Neff_critique05,kalamajska2003new,kalamajska2015}, i.e.\ that there exists a function $W:\R^{2\times2}\to\R$ which is rank-one convex but not quasiconvex.
A number of possible candidates have already been proposed: for example, the function (see \cite{volberg2012ahlfors}) $W^{\#}:\mathbb{R}^{2\times2}\rightarrow \mathbb{R}$ with
 \begin{align}
 W^{\#}(F)&=\left\{\begin{array}{lll}
 -4\, \det F & \text{if }\; \sqrt{\|F\|^2-2\, \det F}+\sqrt{\|F\|^2+2\, \det F}\leq 1,\\
 2\sqrt{\|F\|^2-2\, \det F}-1&\text{otherwise,}
 \end{array}\right.\label{eq:possibleCounterExample}\\[.5em]
 &=\left\{\begin{array}{lll}
 -4\, \lambda_{\rm min}\,\lambda_{\rm max} & \text{if }\; \lambda_{\rm max}\leq \frac{1}{2},\\
 2\,(\lambda_{\rm max}-\lambda_{\rm min})-1 \qquad&\text{otherwise,}
 \end{array}\right.\nonumber
 \end{align}
where $\lambda_{\rm min},\lambda_{\rm max}$ denote the smallest and the largest singular value of $F$, respectively, is known to be rank-one convex\footnote{This follows from the convexity of the function $\lambda_{\rm max}\pm\lambda_{\rm min}=\sqrt{\|F\|^2\pm 2\, \det F}=\sqrt{(F_{11}\pm F_{22})^2(F_{21}\mp F_{12})^2}$, see \cite[Lemma 2.2]{conti2003polyconvexity}. In \cite[Remark 1]{conti2003polyconvexity} it is also noticed that  any SO(2)-invariant polyconvex function can be written as supremum of linear
combinations of the functions $\varphi_c^{\pm}=\lambda_{\rm max} \pm \lambda_{\rm min}-\frac{\lambda_{\rm max}\lambda_{\rm min}}{c}$, for $c\in \mathbb{R}\setminus\{0\}$, $\varphi_0^{\pm}=-\lambda_{\rm max}\, \lambda_{\rm min}$,  by writing it first as supremum of polyaffine functions and
then exploiting $\SO(2)$-invariance. Thus the individual branches of $W^{\#}$ are polyconvex.}, but it is not known whether this function is quasiconvex at $F=0$.

There are, however, a number of special cases for which the two convexity conditions are, in fact, equivalent: %
for example, every quasiconvex quadratic form is polyconvex \cite{terpstra1939darstellung,serre1983,marcellini1984quasiconvex,Sverak92} and, as M\"uller \cite{muller1999rank} has shown, rank-one convexity implies quasiconvexity in dimension two on diagonal matrices \cite{conti2003polyconvexity,chaudhuri2003rank,conti2008quasiconvex}. Moreover, Ball and Murat \cite{Ball84b} showed that every energy function $W:\R^{2\times2}\to\R$ of the form $W(F)=\norm{F}^\alpha + h(\det F)$ with a function $h:\R\to\R$ and $1\leq\alpha<2$ is polyconvex if and only if it is rank-one convex. Iwaniec et al.\ even conjectured that \enquote{\emph{continuous rank-one convex functions $W:\R^{2\times2}\to\R$ are quasiconvex}} \cite[Conjecture 1.1]{astala2012quasiconformal} in general\footnote{Interestingly, the related (but not equivalent) question whether isotropic rank-one convex sets in $\R^{2\times2}$ are already quasiconvex has a positive answer \cite{kreiner2006topology,heinz2015quasiconvexity}.} (whereas Pedregal found \enquote{\emph{some evidence in favour}} \cite{pedregal2014some} of the hypothesis that the two conditions are not equivalent \cite{pedregal1996some}).

In this spirit, we present another condition under which rank-one convexity implies polyconvexity (and thus quasiconvexity), thereby further complicating the search for a counterexample:
we show that any function $W:\GLpz\to\R$ which is isotropic and objective (i.e.\ bi-$\SO(2)$-invariant) as well as \emph{isochoric} is rank-one convex if and only if it is polyconvex.
A function $W:\GLpz\to\R$ is called isochoric\footnote{%
In elasticity theory, isochoric energy functions measure only the \emph{change of form} of an elastic body, not the \emph{change of size}.
For more general elastic energy functions $W\colon\GLpz\to\R$, an additive \emph{isochoric-volumetric split} \cite{ndanou2014hyperbolicity} of the form
\[
	W(F) = W^{\mathrm{iso}}(F) + W^{\mathrm{vol}}(\det F) = W^{\mathrm{iso}}\left(\frac{F}{(\det F)^{\ahalf}}\right) + W^{\mathrm{vol}}(\det F)
\]
into an isochoric part $W^{\mathrm{iso}}\colon\GLpz\to\R$ and a volumetric part $W^{\mathrm{vol}}:\R^+\to\R$ is sometimes assumed, cf.\ Section \ref{sectionContains:applicationToIsoVolSplit}.}
if
\[
	W(a\,F)=W(F) \quad\text{ for all }\; a\in\R^+\colonequals(0,\infty)\,.
\]

Note carefully that we explicitly consider functions which are defined only on $\GLpz$, and not on all of $\R^{2\times2}$. Such a function $W$ can equivalently be expressed as a (discontinuous) function $W:\R^{2\times2}\to\R\cup\{+\infty\}$ with $W(F)=+\infty$ for all $F\notin\GLpz$. In many fields, these energy functions are more suitable for applications than finite-valued functions on $\R^{2\times2}$. In the theory of nonlinear hyperelasticity, for example, the requirement $W(F)\to\infty$ as $\det F\to0$ is commonly assumed to hold. The left- and right-$\SO(2)$-invariance is also motivated by applications in nonlinear elasticity and corresponds to the requirements of objectivity and isotropy, respectively.\footnote{If functions on $\R^{2\times2}$ are considered, then the isotropy requirement is oftentimes assumed to be right-$\OO(2)$-invariance, whereas right-$\SO(2)$-invariance is the natural isotropy condition for functions on $\GLpz$.} While Morrey's conjecture is usually stated for finite-valued functions on all of $\R^{2\times2}$ only, energy functions on $\GLpz$ have long been a valuable source of inspiring examples; indeed, for $n>2$, an early example of a non-continuous function mapping $\R^{n\times n}$ to $\R\cup\{+\infty\}$ which is rank-one convex but not quasiconvex was given by Ball \cite{ball1990sets}, even before \v{S}ver\'ak \cite{Sverak92} found a continuous finite-valued counterexample.
Additional conditions for rank-one convexity of objective and isotropic energy functions on $\GLpz$ have also been considered by \v{S}ilhav\'y \cite{silhavy2002convexity}, Parry and \v{S}ilhav\'y \cite{parry2000rank}, Aubert \cite{aubert1995necessary} and Davies \cite{davies1991simple}.

Note also that a function $W:\GLpz\to\R$ is isotropic, objective and isochoric if and only if $W$ is (left- and right-) \emph{conformally invariant}, i.e.\ $W(A\,F\,B)=W(F)$ for all $A,B\in\CSO(2)$, where
\[
	\CSO(2)=\R^+\cdot\SO(2)=\{a\,Q\in\GLpz \,|\, a\in\R^+\,,\; Q\in\SO(2)\}
\]
denotes the \emph{conformal special orthogonal group}. In the literature, one also encounters the concept of \emph{conformal energies} \cite{yan1997rank}, which are functions $W$ such that $W(F)$ vanishes if and only if $F\in\CSO(2)$, e.g.\ $W(F)=\norm{F}^2-2\det F$. However, as this example shows, such energies are generally not isochoric (or conformally invariant).

The idea of finding new isochoric functions which are rank-one convex has arisen from the search for a function of the isotropic invariants $\|\dev_2 \log U\|^2$ and $[\tr (\log U)]^2$ of the logarithmic strain tensor $\log U$ which is rank-one convex or polyconvex (see \cite{neff2013hencky,neff2013hencky,Neff_Osterbrink_Martin_hencky13,NeffEidelOsterbrinkMartin_Riemannianapproach,criscione2000invariant,Walton05}), since the commonly used quadratic Hencky energy
\[
	W_{_{\rm H}}(F) = W_{_{\rm H}}^{\mathrm{iso}}\left(\frac {F}{(\det F)^{\ahalf}}\right)+W_{_{\rm H}}^{\mathrm{vol}}(\det F) = \mu \,\|\dev_2\log U\|^2+ \frac{\kappa}{2}\,[\tr (\log U)]^2
\]
is not rank-one convex even in $\SL(2)\colonequals\{X\in \GLpz\;|\det{X}=1\}$, see \cite{NeffGhibaLankeit}. Here, $\mu>0$ is the infinitesimal shear modulus,
$\kappa=\frac{2\mu+3\lambda}{3}>0$ is the infinitesimal bulk modulus, $\lambda$ is the first Lam\'{e} constant, $F=\nabla \varphi$ is the gradient of deformation, $U=\sqrt{F^T F}$ is the right stretch tensor and $\log U$ denotes the principal matrix logarithm of $U$. For $X\in\R^{2\times 2}$, we denote by $\|{X}\|$ the Frobenius tensor norm, $\tr{(X)}$ is the trace of $X$, $\dev_2 X =X-\frac{1}{2}\,\tr(X)\cdot\id$ is the deviatoric part of $X$ and $\id$ denotes the identity tensor on $\R^{2\times 2}$.

Promising candidates for a polyconvex formulation in terms of $\|\dev_2 \log U\|^2$ and $[\tr (\log U)]^2$ are the exponentiated Hencky energies previously considered in a   series of papers \cite{NeffGhibaLankeit,NeffGhibaPoly,NeffGhibaPlasticity,ghiba2015exponentiated}:
\begin{align}\label{thdefHen}\hspace{-2mm}
 W_{_{\rm eH}}(F) = W_{_{\rm eH}}^{\mathrm{iso}}\left(\frac {F}{(\det F)^{\ahalf}}\right)+W_{_{\rm eH}}^{\mathrm{vol}}(\det F) = \dd\frac{\mu}{k}\,e^{k\,\|\dev_2\log U\|^2}+\frac{\kappa}{2\,\widehat{k}}\,e^{\widehat{k}\,[(\log \det U)]^2}\,,
\end{align}
where  $k,\widehat{k}$ are additional dimensionless
parameters.

\section{Preliminaries}\setcounter{equation}{0}
In order to establish our main result, i.e.\ that rank-one convexity and polyconvexity are equivalent for isochoric energy functions, we first need to recall some conditions for these convexity properties. %
In the following, we will assume $W:\GLpz\to\R,\; F\mapsto W(F)$ to be an objective, isotropic function. It is well known that such a function can be expressed in terms of the singular values of $F$: there exists a uniquely determined function $g:\R^+\times\R^+\to\R$ such that
\begin{equation}
\label{eq:gDefinition}
	W(F) = g(\lambda_1,\, \lambda_2)
\end{equation}
for all $F\in\GLpz$ with singular values $\lambda_1,\lambda_2$. Note that the isotropy of $W$ also implies the symmetry condition $g(\lambda_1,\lambda_2)=g(\lambda_2,\lambda_1)$.

\subsection{A sufficient condition for polyconvexity}

A proof of the following lemmas can be found in \cite{ghiba2015exponentiated}.
\begin{lemma}
If $\, Y:[1,\infty)\rightarrow\mathbb{R}$ is non-decreasing and convex and $Z:\GLpz\rightarrow[1,\infty)$ is polyconvex, then $Y\circ Z$ is polyconvex.
\end{lemma}

\begin{lemma}
The function $Z:\GLpz\to[1,\infty)$ with $Z(F)=\frac{\norm[\mathrm{op}]{F}^2}{\det F}$, where $\norm[\mathrm{op}]{F}=\max\{\lambda_1,\,\lambda_2\}$ denotes the spectral norm of $F\in\GLpz$ with singular values $\lambda_1,\lambda_2$, is polyconvex on $\GLpz$. Note that the function $Z$ can be expressed as $Z(F)=g(\lambda_1,\lambda_2)$ with $g(\lambda_1\,\lambda_2) = \frac{\max\{\lambda_1^2,\lambda_2^2\}}{\lambda_1\,\lambda_2}$.
\end{lemma}
These two lemmas immediately imply the next proposition \cite{ghiba2015exponentiated}, which will play a key role in showing that isochoric, rank-one convex energies are already polyconvex.
\begin{proposition}\label{polysilhavy}
If, for given $W:\GLpz\rightarrow \mathbb{R}$, there exists a non-decreasing and convex function $\, h:[1,\infty)\rightarrow\mathbb{R}$ such that  $W=h\circ Z$, where $Z(F)=\frac{\norm[\mathrm{op}]{F}^2}{\det F}$, then $W$ is polyconvex.
\end{proposition}

\subsection{A necessary condition for rank-one convexity}
We prove the following well-known necessary condition for rank-one convexity:
\begin{lemma}
\label{lemma:rocImpliesSeparateConvexity}
Let $W:\GLpz\to\R$ be objective, isotropic and rank-one convex, and let $g:\R^+\times\R^+\to\R$ denote the representation of $W$ in terms of singular values. Then $g$ is separately convex, i.e.\ the mapping $\lambda_1\mapsto g(\lambda_1,\lambda_2)$ is convex for fixed $\lambda_2\in\R^+$ and the mapping $\lambda_2\mapsto g(\lambda_1,\lambda_2)$ is convex for fixed $\lambda_1\in\R^+$.
\end{lemma}
\begin{proof}
For $a,b\in\R$, we define
\[
	\diag(a,b) \colonequals \begin{pmatrix}a&0\\0&b\end{pmatrix}\,.
\]
Let $\lambda_2\in\R^+$ be fixed. Since the matrix $\diag(1,0)$ has rank one, the rank-one convexity of $W$ implies that the mapping
\[
	t\mapsto W(\diag(1,\lambda_2) + t\cdot\diag(1,0)) = W(\diag(1+t,\lambda_2)) = g(1+t,\lambda_2)\,,\qquad t\in(-1,\infty)\,,
\]
is convex. Therefore, the function $g$ is convex in the first component and, for symmetry reasons, convex in the second component.
\end{proof}
\noindent Note that for an energy function $W$ of class $C^2$, the separate convexity of $g$ is equivalent to the \emph{tension-extension inequalities (TE-inequalities)}
\[
	\frac{\partial^2 g}{\partial \lambda_1^2 }\geq 0 \qquad\text{and }\qquad \frac{\partial^2 g}{\partial \lambda_2^2 }\geq 0 \qquad\text{for }\lambda_1,\lambda_2\in\R^+\,.
\]

\section{The equivalence of rank-one convexity and polyconvexity for isochoric energy functions}
\label{section:mainResult}
\subsection{The main result}
We now focus on isochoric functions $W$ on $\GLpz$, i.e.\ functions which satisfy $W(a\,F)=W(F)$ for all $F\in\GLpz$ and all $a>0$. These functions can be uniquely represented in terms of the ratio $\frac{\lambda_1}{\lambda_2}$ of the singular values of $F$.
\begin{lemma}
\label{lemma:fracRepresentation}
Let $W:\GLpz\to\R,\;F\mapsto W(F)$ be an objective, isotropic function which is additionally \emph{isochoric}, i.e.\ satisfies $W(a\,F)=W(F)$ for all $F\in\GLpz$ and all $a>0$. Then there exists a unique function $h:\R^+\to\R$ with $h(t)=h(\frac1t)$ such that $W(F)=h\big(\frac{\lambda_1}{\lambda_2}\big)$ for all $F\in\GLpz$ with singular values $\lambda_1,\lambda_2\in\R^+$.
\end{lemma}
\begin{remark}
Note that Lemma \ref{lemma:fracRepresentation} explicitly requires $W$ to be defined on $\GLpz$ only: for functions on all of $\GL(2)$, the isotropy requirement must be extended from right-$\SO(2)$-invariance to right-$\OO(2)$-invariance in order to ensure a representation in terms of the singular values; if singular matrices are included in the domain of $W$, then $h$ is not well defined in the form stated in the lemma.
\end{remark}
\begin{proof}
Since $W$ is objective and isotropic, there exists a function $g:\R^+\times\R^+\to\R$ with $W(F)=g(\lambda_1,\lambda_2)=g(\lambda_2,\lambda_1)$ for all $F\in\GLpz$, where $\lambda_1,\lambda_2$ are the singular values of $F$. Then
\[
	W(F)=W\left(\frac{F}{\sqrt{\det F}}\right) = g\left(\frac{\lambda_1}{\sqrt{\lambda_1\,\lambda_2}},\,\frac{\lambda_2}{\sqrt{\lambda_1\,\lambda_2}}\right) = g\left(\sqrt{\frac{\lambda_1}{\lambda_2}},\,\sqrt{\frac{\lambda_2}{\lambda_1}}\right)\,,
\]
hence for $h:\R^+\to\R$ with $h(t) \colonequals g\left(\sqrt{t},\frac{1}{\sqrt{t}}\right)$ we find
\[
	h\left(\frac{\lambda_1}{\lambda_2}\right) = g\left(\sqrt{\frac{\lambda_1}{\lambda_2}},\frac{1}{\sqrt{\frac{\lambda_1}{\lambda_2}}}\right) = g\left(\sqrt{\frac{\lambda_1}{\lambda_2}},\sqrt{\frac{\lambda_2}{\lambda_1}}\right) = W(F)\,,
\]
and the symmetry of $g$ (which follows from the isotropy of $W$) implies
\[
	h(t) = g\left(\sqrt{t},\frac{1}{\sqrt{t}}\right) = g\left(\frac{1}{\sqrt{t}},\sqrt{t}\right) = g\left(\sqrt{\frac{1}{t}},\sqrt{\frac{1}{\frac1t}}\right) = h\left(\frac1t\right)\,.
\]
Finally, the uniqueness of $h$ follows directly from the equality $h(t)=W(\diag(t,1))$.
\end{proof}
We are now ready to prove our main result.
\begin{theorem}
\label{theorem:mainResult}
Let $W:\GLpz\to\R,\;\;F\mapsto W(F)$ be an objective, isotropic and isochoric function, and let \hbox{$h:\R^+\to\R$,\; $g:\R^+\times\R^+\to\R$} denote the uniquely determined functions with
\[
	W(F)=g(\lambda_1,\lambda_2)=h\left(\frac{\lambda_1}{\lambda_2}\right)=h\left(\frac{\lambda_2}{\lambda_1}\right)
\]
for all $F\in\GLpz$ with singular values $\lambda_1,\lambda_2$. Then the following are equivalent:
\begin{itemize}
\item[i)] $W$ is polyconvex.
\item[ii)] $W$ is rank-one convex,
\item[iii)] $g$ is separately convex,
\item[iv)] $h$ is convex on $\R^+$,
\item[v)] $h$ is convex and non-decreasing on $[1,\infty)$.
\end{itemize}
\end{theorem}
\begin{proof}
The implication i) $\Rightarrow$ ii) is well-known to hold in general, whereas the implication ii) $\Rightarrow$ iii) is stated in Lemma \ref{lemma:rocImpliesSeparateConvexity}.\\
iii) $\Rightarrow$ iv):
If $g$ is separately convex, then the mapping
\[
	\lambda_1 \mapsto g(\lambda_1,1) = h(\lambda_1)
\]
is convex, thus $h$ is convex on $\R^+$.\\
iv) $\Rightarrow$ v):
Assume that $h$ is convex on $\R^+$. Then, of course, $h$ is also convex on $[1,\infty)$, and it remains to show the monotonicity of $h$. Let $1\leq t_1<t_2$. Then $\frac{1}{t_2}<1\leq t_1<t_2$, i.e.\ $t_1$ lies in the convex hull of $\frac{1}{t_2}$ and $t_2$. But then $t_1 = s\,\frac{1}{t_2} + (1-s)\,t_2$ for some $s\in(0,1)$, and thus the convexity of $h$ on $\R^+$ implies that
\[
	h(t_1) = h\left(s\,\frac{1}{t_2} + (1-s)\,t_2\right) \leq s\,h\!\left(\frac{1}{t_2}\right) + (1-s)\,h(t_2) = s\,h(t_2) + (1-s)\,h(t_2) = h(t_2)\,,
\]
hence $h$ is non-decreasing on $[1,\infty)$.\\
iv) $\Rightarrow$ v):
Assume that $h$ is convex and non-decreasing on $[1,\infty)$. Then we can apply Proposition \ref{polysilhavy}: since the mapping
\[
	F \mapsto \frac{\norm[\mathrm{op}]{F}^2}{\det F} = \frac{\max\{\lambda_1^2,\lambda_2^2\}}{\lambda_1\,\lambda_2} \in [1,\infty)
\]
is polyconvex \cite{ghiba2015exponentiated} and $h$ is convex and non-decreasing on $[1,\infty)$, the mapping
\[
	F \mapsto h\left(\frac{\max\{\lambda_1^2,\lambda_2^2\}}{\lambda_1\,\lambda_2}\right) = h\left(\frac{\lambda_1}{\lambda_2}\right) = W(F)
\]
is polyconvex as well.
\end{proof}
If the function $h$ is continuously differentiable, then the criteria in Theorem \ref{theorem:mainResult} can be simplified even further.
\begin{corollary}
\label{cor:mainResultForDifferentiableFunctions}
Let $W:\GLpz\to\R$ be an objective, isotropic and isochoric function, and let $h:\R^+\to\R$ denote the uniquely determined function with $W(F)=h\big(\frac{\lambda_1}{\lambda_2}\big)$ for all $F\in\GLpz$ with singular values $\lambda_1,\lambda_2$. If $h\in C^1(\R^+)$, then $W$ is polyconvex if and only if $h$ is convex on $[1,\infty)$.
\end{corollary}
\begin{proof}
We only need to show that the stated criterion is sufficient for the polyconvexity of $W$. Assume therefore that $h$ is convex on $[1,\infty)$. Taking the derivative on both sides of the equality $h(t)=h\big(\frac1t\big)$, which holds for all $t\in\R^+$, yields
\[
	h^{\prime}(t) = -\frac{1}{t^2}\cdot h^{\prime}\left(\frac1t\right)\,.
\]
In particular, $h^{\prime}(1)=-h^{\prime}(1)$ and thus $h^{\prime}(1)=0$. Since the convexity of $h$ implies the monotonicity of $h^{\prime}$ on $[1,\infty)$, we find $h^{\prime}(t)\geq0$ for all $t\in[1,\infty)$. This means that $h$ is non-decreasing on $[1,\infty)$, and applying criterion v) in Theorem \ref{theorem:mainResult} yields the polyconvexity of $W$.
\end{proof}

\section{Criteria for rank-one convexity and polyconvexity in terms of different energy representations}
\label{section:criteriaDifferentRepresentations}
\subsection{Energy functions in terms of the logarithmic strain}
We will now assume that the function $W$ is of class $C^2$. While the criterion $h^{\prime\prime}(t)\geq0$ for all $t\in[1,\infty)$ in Corollary \ref{cor:mainResultForDifferentiableFunctions} is easy to state, isochoric elastic energy functions in nonlinear hyperelasticity are typically not immediately given in terms of the quantity $\frac{\lambda_1}{\lambda_2}$. We therefore consider different representations of such functions in our search for easily verifiable polyconvexity criteria.
\begin{lemma}
\label{lemma:logSquaredRepresentation}
Let $W:\GLpz\to\R$ be objective, isotropic and isochoric. Then there exist unique functions $f,\ftilde:[0,\infty)\to\R$ such that
\begin{align*}
	\text{i)} \quad W(F) &= f\left(\log^2\frac{\lambda_1}{\lambda_2}\right)\,,\\
	\text{ii)} \quad W(F) &= \ftilde(\norm{\dev_2\log U}^2)
\end{align*}
for all $F\in\GLpz$, where $\lambda_1,\lambda_2$ denote the singular values of $F$,\; $U=\sqrt{F^TF}$ is the positive definite symmetric polar factor in the right polar decomposition of $F$, $\dev_2 X = X - \frac{\tr(X)}{2}\cdot\id$ is the deviatoric part of $X\in\R^{2\times 2}$, $\log$ denotes the principal matrix logarithm on $\PSym(2)$ and $\norm{\,.\,}$ is the Frobenius matrix norm.
\end{lemma}
\begin{proof}
i):\; Let us first recall that from Lemma \ref{lemma:fracRepresentation} that there exists a unique function $h:\mathbb{R}^+\rightarrow\mathbb{R}^+$ such that $W(F)=h\big(\frac{\lambda_1}{\lambda_2}\big)$ for all $F\in\GLpz$ with singular values $\lambda_1,\lambda_2$.
Let $f(\theta)=h(e^{\sqrt{\theta}})$ for $\theta>0$. Since
\[
	\sqrt{\log^2\frac{\lambda_1}{\lambda_2}} = \left|\log\frac{\lambda_1}{\lambda_2}\right| = \log\,\frac{\max\{\lambda_1,\lambda_2\}}{\min\{\lambda_1,\lambda_2\}}\,,
\]
we find
\[
	f\left(\log^2\frac{\lambda_1}{\lambda_2}\right) = h\left(e^{\sqrt{\log^2\frac{\lambda_1}{\lambda_2}}}\right) = h\left(e^{\log\frac{\max\{\lambda_1,\lambda_2\}}{\min\{\lambda_1,\lambda_2\}}}\right) = h\left(\frac{\max\{\lambda_1,\lambda_2\}}{\min\{\lambda_1,\lambda_2\}}\right) = h\left(\frac{\lambda_1}{\lambda_2}\right) = W(F)
\]
for all $F\in\GLpz$ with singular values $\lambda_1,\lambda_2$. To show the uniqueness of $f$, we simply note that
\[
	f(\theta) = f\left(\log^2\left(\frac{e^{\sqrt{\theta}}}{1}\right)\right) = W(\diag(e^{\sqrt{\theta}},\,1))
\]
for all $\theta>0$.
ii):\; It has been previously shown \cite{NeffGhibaLankeit} that
\[
	\|\dev_2\log U\|^2=\frac{1}{2}\,\log^2 \frac{\lambda_1}{\lambda_2}\,.
\]
The equality $W(F) = \ftilde(\norm{\dev_2\log U}^2)$ is therefore satisfied for all $F\in\GLpz$ if and only if $\ftilde(t) = f(2\,t)$, where $f$ is given by i).
\end{proof}
Note carefully that for $n>2$, not every objective, isotropic and isochoric energy $W:\GLp(n)\to\R$ can be written in terms of $\norm{\dev_n\log U}^2$ in the way Lemma \ref{lemma:logSquaredRepresentation} states for $n=2$. However, there always exists a function $\widehat{W}:\Sym(n)\to\R$ such that $W(F)=\widehat{W}(\dev_n\log U)$ for all $F\in\GLp(n)$ with $U=\sqrt{F^TF}$.\\[1em]
We can now state Theorem \ref{theorem:mainResult} in terms of the functions $f,\ftilde$ as defined in Lemma \ref{lemma:logSquaredRepresentation}.

\begin{proposition}
\label{prop:mainResultInTermsOfLogSquared}
Let $W:\GLpz\to\R,\;F\mapsto W(F)$ be an objective, isotropic and isochoric function and let $f,\ftilde:[0,\infty)\to\R$ denote the uniquely determined functions with
\[
	W(F)=\ftilde(\norm{\dev_2\log U}^2)=f\left(\log^2\frac{\lambda_1}{\lambda_2}\right)
\]
for all $F\in\GLpz$ with singular values $\lambda_1,\lambda_2$. If $f,\ftilde\in C^2([0,\infty))$, then the following are equivalent:
\begin{itemize}
\item[i)] $W$ is polyconvex,
\item[ii)] $W$ is rank-one convex,
\item[iii)] $2\,\theta\,f^{\prime\prime}(\theta)+ (1-\sqrt{\theta})\,f^{\prime}(\theta)\geq 0$ \quad for all $\theta\in(0,\infty)$,
\item[iv)] $2\,\eta\,\ftilde^{\prime\prime}(\eta)+ (1-\sqrt{2\,\eta})\,\ftilde^{\prime}(\eta)\geq 0$ \quad for all $\eta\in(0,\infty)$.
\end{itemize}
\end{proposition}
\begin{proof}
For $h:\R^+\to\R$ with $h(t) = f(\log^2 t)$ we find
\[
	h\left(\frac{\lambda_1}{\lambda_2}\right) = f\left(\log^2\frac{\lambda_1}{\lambda_2}\right) = W(F)
\]
for all $F\in\GLpz$ with singular values $\lambda_1,\lambda_2$. If $f\in C^2([0,\infty))$, then $h\in C^2(\R^+)$, thus we can apply Corollary \ref{cor:mainResultForDifferentiableFunctions} to find that $W$ is polyconvex (and, equivalently, rank-one convex) if and only if $h$ is convex on $[1,\infty)$. Since $h^{\prime\prime}$ is continuous on $\R^+$, this convexity of $h$ is equivalent to $h^{\prime\prime}(t)\geq0$ for all $t\in(1,\infty)$.
We compute
\begin{equation}
\label{eq:hPrimeITOf}
	h^{\prime}(t) = 2\,f^{\prime}(\log^2t)\cdot\frac{\log t}{t}
\end{equation}
as well as
\begin{align*}
	h^{\prime\prime}(t) &= 4\,f^{\prime\prime}(\log^2t)\cdot\frac{\log^2t}{t^2} - 2\,f^{\prime}(\log^2t)\cdot \frac{\log t}{t^2} + 2\,f^{\prime}(\log^2t)\cdot \frac{1}{t^2}\\
	&= \frac{2}{t^2}\,\Big(2\,(\log^2t)\, f^{\prime\prime}\big(\log^2t\big) + (1-\log t)\, f^{\prime}\big(\log^2t\big)\Big)\,.
\end{align*}
Writing $t>1$ as $t=e^{\sqrt{\theta}}$ with $\theta>0$ we find
\[
	h^{\prime\prime}(t) = \frac{2}{e^{2\sqrt{\theta}}}\,\Big(2\,\theta\,f^{\prime\prime}(\theta) + (1-\sqrt{\theta})\,f^{\prime}(\theta)\Big)\,.
\]
Since the mapping $\theta\to e^{\sqrt{\theta}}$ is bijective from $(0,\infty)$ to $(1,\infty)$, the condition
\begin{equation}
\label{eq:convexityInTermsOfFrac}
	h^{\prime\prime}(t)\geq0 \quad\text{ for all } t\in(1,\infty)
\end{equation}
is therefore equivalent to
\begin{equation}
\label{eq:convexityInTermsOfLogSquaredFrac}
	2\,\theta\,f^{\prime\prime}(\theta) + (1-\sqrt{\theta})\,f^{\prime}(\theta)\geq0 \quad\text{ for all } \theta\in(0,\infty)\,,
\end{equation}
which is exactly criterion iii).\\
It remains to show that iii) and iv) are equivalent. Since $\ftilde(\eta)=f(2\,\eta)$ \;(see Lemma \ref{lemma:logSquaredRepresentation}), we find
\[
	2\,\eta\,\ftilde^{\prime\prime}(\eta)+ \left(1-\sqrt{2\,\eta}\,\ftilde^{\prime}(\eta)\right) = 2\,\left[2\,(2\,\eta)\,f^{\prime\prime}(2\,\eta)+ \left(1-\sqrt{(2\,\eta)}\right)\,f^{\prime}(2\,\eta)\right]\,.
\]
Thus iv) is satisfied for all $\eta\in\R^+$ if and only if iii) is satisfied for all $\theta=2\,\eta\in\R^+$.
\end{proof}
\noindent In addition to criterion iii) in Proposition \ref{prop:mainResultInTermsOfLogSquared}, the polyconvexity of $W$ also implies the monotonicity of $f$:
\begin{corollary}
\label{cor:polyconvexityImpliesMonotonicityITOf}
Under the assumptions of Proposition \ref{prop:mainResultInTermsOfLogSquared}, if $W$ is polyconvex (or, equivalently, rank-one convex), then $f^{\prime}(\theta)\geq0$ for all $\theta>0$.
\end{corollary}
\begin{proof}
According to Theorem \ref{theorem:mainResult}, the polyconvexity of $W$ implies that $h=f\circ\log^2$ is non-decreasing on $[1,\infty)$. Then
\[
	0 \leq h^{\prime}(t) = 2\,f^{\prime}(\log^2t)\cdot\frac{\log t}{t}
\]
for all $t>1$ and thus $0 \leq f^{\prime}(\log^2t)$ for all $t>1$, which immediately implies $f^{\prime}(\theta)\geq0$ for all $\theta>0$.
\end{proof}

\subsection{Energy functions in terms of the distortion function}
We now consider the representation of an isochoric energy $W(F)$ in terms of $\K(F)=\frac12\,\frac{\norm{F}^2}{\det F}$, where $\norm{\,.\,}$ denotes the Frobenius matrix norm; the mapping $\K$ is also known as the (planar) \emph{distortion function} \cite{iwaniec2009} or \emph{outer distortion} \cite[eq.\ (14)]{iwaniec2011invitation}. Note that $\K\geq1$ and that, for $F\in\GLpz$, $\K(F)=1$ if and only if $F$ is conformal, i.e.\ if $F=a\cdot R$ with $a\in\R^+$ and $R\in\SO(2)$.
In the two-dimensional case, every objective, isotropic and isochoric (i.e.\ conformally invariant) energy can be written in terms of $\K$.
\begin{lemma}
\label{lemma:Krepresentation}
Let $W:\GLpz\to\R$ be objective, isotropic and isochoric. Then there exists a unique function $\z:[1,\infty)\to\R$ with
\[
	W(F) = \z\left(\K(F)\right) = \z\left(\frac12\,\frac{\norm{F}^2}{\det F}\right)
\]
for all $F\in\GLpz$.
\end{lemma}
\begin{proof}
It can easily be seen that the function $p:[1,\infty)\to[1,\infty)$ with $p(t)=\frac12\,(t+\frac1t)$ is bijective, and that its inverse is given by
\[
	q(s)=p^{-1}(s) = y+\sqrt{y^2-1}\,.
\]
Then $q(\frac12(t+\frac1t))=t$ for all $t\in[1,\infty)$, while for $t\in(0,1)$ we find
\[
	q\Big(\frac12\left(t+\frac1t\right)\Big) = q\Big(\,\frac12\left(\smash{\,\overbrace{\frac1t}^{>1}}+\frac{1}{\frac1t}\,\right)\,\Big) = \frac1t\,.\vphantom{\overbrace{\frac1t}^{>1}}
\]
Therefore $q(\frac12(t+\frac1t)) = \max\{t,\frac1t\}$ for all $t\in\R^+=(0,\infty)$.

According to Lemma \ref{lemma:fracRepresentation}, there exists a unique function $h:\mathbb{R}^+\rightarrow\mathbb{R}^+$ such that
$W(F)=h\big(\frac{\lambda_1}{\lambda_2}\big) = h\big(\frac{\lambda_2}{\lambda_1}\big)$ %
for all $F\in\GLpz$ with singular values $\lambda_1,\lambda_2$. Then the function $z\colonequals h\circ q$ has the desired property: since
\[
	\frac12\,\frac{\norm{F}^2}{\det F} = \frac12\,\frac{\lambda_1^2+\lambda_2^2}{\lambda_1\,\lambda_2} = \frac12\left(\frac{\lambda_1}{\lambda_2} + \frac{\lambda_2}{\lambda_1}\right)\,,
\]
we find
\[
	z\bigg(\frac12\,\frac{\norm{F}^2}{\det F}\bigg) = h\Big(q\bigg(\frac12\,\frac{\norm{F}^2}{\det F}\bigg)\Big) = h\Big(q\bigg(\frac12\bigg(\frac{\lambda_1}{\lambda_2} + \frac{1}{\frac{\lambda_1}{\lambda_2}}\bigg)\bigg)\Big) = h\bigg(\max\left\{\frac{\lambda_1}{\lambda_2}\,,\; \frac{\lambda_2}{\lambda_1}\right\}\bigg) = W(F)\,.
\]
The uniqueness follows directly from the observation that
\[
	\z(r) = W\Big(\!\diag\big(r+\sqrt{r^2-1},\;1\big)\Big)
\]
for all $r\in[1,\infty)$.
\end{proof}
By means of this representation formula, we can easily show that every objective, isotropic and isochoric function on $\GLpz$ satisfies the \emph{tension-compression symmetry} condition $W(F^{-1})=W(F)$: since
\[
	F^{-1} = \frac{1}{\det F}\, \matr{F_{22}&-F_{12}\\-F_{21}&F_{11}} \quad\text{ for }\; F = \matr{F_{11}&F_{12}\\F_{21}&F_{22}}\,,
\]
we find
\begin{align*}
	\K(F^{-1}) \;=\; \frac12\,\frac{\norm{F^{-1}}^2}{\det (F^{-1})} \;&=\; \frac{\det F}{2}\,\left\|\frac{1}{\det F}\, \matr{F_{22}&-F_{12}\\-F_{21}&F_{11}}\right\|^2\\
	&=\; \frac{1}{2\,\det F}\,\left\|\matr{F_{22}&-F_{12}\\-F_{21}&F_{11}}\right\|^2 \;=\; \frac12\,\frac{\norm{F}^2}{\det F} \;=\; \K(F)
\end{align*}
and thus $W(F)=\z(\K(F))=\z(\K(F^{-1}))=W(F^{-1})$ for all $F\in\GLpz$. Note that this implication is restricted to the two-dimensional case: isochoric energy functions on $\GLp(n)$ are generally not tension-compression symmetric for $n>2$.\\[1em]
\noindent Criteria for the polyconvexity of $W$ can now be established in terms of the function $\z$ corresponding to $W$.
\begin{proposition}
\label{prop:mainResultInTermsOfK}
Let $W:\GLpz\to\R,\;F\mapsto W(F)$ be an objective, isotropic and isochoric function and let $\z:[1,\infty)\to\R$ denote the uniquely determined function with
\[
	W(F) = \z\left(\K(F)\right) = \z\left(\frac12\,\frac{\norm{F}^2}{\det F}\right)
\]
for all $F\in\GLpz$. If $\z\in C^2([0,\infty))$, then the following are equivalent:
\begin{itemize}
\item[i)] $W$ is polyconvex,
\item[ii)] $W$ is rank-one convex,
\item[iii)] $(r^2-1)\,(r+\sqrt{r^2-1})\,\z^{\prime\prime}(r) + z^{\prime}(r)\geq 0$ \quad for all $r\in(1,\infty)$.
\end{itemize}
\end{proposition}
\begin{proof}
As indicated in the proof of Lemma \ref{lemma:Krepresentation}, the unique function $h$ with
$W(F)=h\big(\frac{\lambda_1}{\lambda_2}\big) = h\big(\frac{\lambda_2}{\lambda_1}\big)$ for all $F\in\GLpz$ with singular values $\lambda_1,\lambda_2$ is given by $h(t)=\z\big(\frac{t}{2}+\frac1{2\,t}\big)$ for all $t\geq1$.
By Corollary \ref{cor:mainResultForDifferentiableFunctions}, we only need to show that condition iii) is equivalent to the convexity of $h$ on $[1,\infty)$, i.e.\ to $h^{\prime\prime}(t)\geq0$ for all $t>1$. For $t>1$, we find $h^{\prime}(t) = \frac12\left(1-\frac{1}{t^2}\right)\cdot \z^{\prime}\left(\frac{t}{2}+\frac1{2\,t}\right)$ and
\begin{align*}
	h^{\prime\prime}(t) &= \frac14\,\Big(1-\frac{1}{t^2}\Big)^2\cdot \z^{\prime\prime}\Big(\frac{t}{2}+\frac1{2\,t}\Big) \,+\, \frac{1}{t^3}\cdot \z^{\prime}\Big(\frac{t}{2}+\frac1{2\,t}\Big)\\
	&= \frac{1}{t^3}\,\left[ \frac{t}{4}\,\Big(t-\frac1t\Big)^2\cdot z^{\prime\prime}\Big(\frac{t}{2}+\frac1{2\,t}\Big) \,+\, z^{\prime}\Big(\frac{t}{2}+\frac1{2\,t}\Big)\right]\\
	&= \frac{1}{t^3}\,\left[ t\,\bigg(\!\Big(\frac{t}{2}+\frac1{2\,t}\Big)^2-1\bigg)\cdot z^{\prime\prime}\Big(\frac{t}{2}+\frac1{2\,t}\Big) \,+\, z^{\prime}\Big(\frac{t}{2}+\frac1{2\,t}\Big) \right]\,,
\end{align*}
thus
\[
	h^{\prime\prime}(t)\geq0 \quad \Longleftrightarrow\quad 0 \;\leq\; t\cdot\bigg(\!\Big(\frac{t}{2}+\frac1{2\,t}\Big)^2-1\bigg)\cdot z^{\prime\prime}\Big(\frac{t}{2}+\frac1{2\,t}\Big) \,+\, z^{\prime}\Big(\frac{t}{2}+\frac1{2\,t}\Big)\,.
\]
Recall from the proof of Lemma \ref{lemma:Krepresentation} that the mapping $r\mapsto q(r) = r+\sqrt{r^2-1}$ bijectively maps $(1,\infty)$ onto itself and that $\frac{q(r)}{2}+\frac{1}{2\,q(r)} = r$ for all $r>1$. Therefore, by writing $t=q(r)$, we find that the inequality $h^{\prime\prime}(t)\geq0$ holds for all $t>1$ if and only if
\begin{align*}
	0 &\leq q(r)\cdot\Bigg(\!\bigg(\frac{q(r)}{2}+\frac1{2\,q(r)}\bigg)^2-1\Bigg)\cdot z^{\prime\prime}\bigg(\frac{q(r)}{2}+\frac1{2\,q(r)}\bigg) \,+\, z^{\prime}\bigg(\frac{q(r)}{2}+\frac1{2\,q(r)}\bigg)\\[.5em]
	&= q(r)\cdot(r^2-1)\cdot z^{\prime\prime}(r) \,+\, z^{\prime}(r) \;=\; (r+\sqrt{r^2-1})\cdot(r^2-1)\cdot z^{\prime\prime}(r) \,+\, z^{\prime}(r)
	\qquad\text{ for all }\; r>1\,.\qedhere
\end{align*}
\end{proof}
\noindent An example for the application of Proposition \ref{prop:mainResultInTermsOfK} can be found in Appendix \ref{appendix:distances}.

\section{Applications}
\subsection{The quadratic and the exponentiated isochoric Hencky energy}
Proposition \ref{prop:mainResultInTermsOfLogSquared} can directly be applied to isochoric energy functions given in terms of $\|\dev_2\log U\|^2$.
\begin{corollary}
\label{cor:henckyExamples}
\begin{itemize}\item[]
\item[i)] The isochoric Hencky energy $\|\dev_2\log U\|^2=\frac{1}{2}\log^2 \frac{\lambda_1}{\lambda_2}$ is not polyconvex and not rank-one convex on $\GLpz$.
\item[ii)] The exponentiated isochoric Hencky energy $e^{k\|\dev_2\log U\|^2}=e^{k\,\|\log \frac{U}{\det U^{\ahalf}}\|^2}=e^{\frac{k}{2}\log ^2 \frac{\lambda_1}{\lambda_2}}$ is rank-one convex (and therefore polyconvex) on $\GLpz$ if and only if $k\geq \frac{1}{4}$.
\end{itemize}
\end{corollary}
\begin{proof}
\noindent i)\; In the case of the isochoric Hencky energy $W(F)=\|\dev_2\log U\|^2=\frac{1}{2}\log^2 \frac{\lambda_1}{\lambda_2}$, the  function $\ftilde$ is defined by $\ftilde(\eta)=\eta$. This function %
does not satisfy condition iv) in Proposition \ref{prop:mainResultInTermsOfLogSquared}: since
    \begin{align}
    	2\,\eta\,\ftilde^{\prime\prime}(\eta)+ (1-\sqrt{2\,\eta})\,\ftilde^{\prime}(\eta) = 1-\sqrt{2\,\eta}\,,
    \end{align}
    the inequality is not satisfied for $\eta>\frac12$.

\noindent ii)\; For the exponentiated isochoric Hencky energy $W(F)=e^{k\|\dev_2\log U\|^2}=e^{k\,\|\log \frac{U}{\det U^{\ahalf}}\|^2}=e^{\frac{k}{2}\log ^2 \frac{\lambda_1}{\lambda_2}}$, the functions $f,\ftilde:[0,\infty)\rightarrow\mathbb{R}$ are defined by $f(\theta)=e^{\frac{k}{2}\theta}$ and $\ftilde(\eta)=e^{k\eta}$. We find
\[
	2\,\eta\,\ftilde^{\prime\prime}(\eta)+(1-\sqrt{2\,\eta})\,\ftilde^{\prime}(\eta) = 2\,\eta\,k^2\,e^{k\eta} + (1-\sqrt{2\,\eta})\,\,k\,e^{k\eta}\,,
\]
thus condition iv) in Proposition \ref{prop:mainResultInTermsOfLogSquared} is equivalent to
\[
	\quad k \geq \frac{\sqrt{2\,\eta}-1}{2\,\eta } \quad\text{ for all }\;\eta>0\,.
\]
This inequality is satisfied if and only if $k\geq \frac{1}{4}$. Therefore the requirement $k\geq \frac{1}{4}$ is necessary and sufficient for the rank-one convexity as well as for the polyconvexity of the isochoric exponentiated Hencky energy $e^{k\|\dev_2\log U\|^2}$.
\end{proof}
\noindent Our results can also be applied to non-isochoric energy functions possessing an additive \emph{isochoric-volumetric split}\footnote{In nonlinear elasticity theory, the assumption that an elastic energy function takes on this specific form is due to the physically plausible requirement that the mean pressure should depend only on the determinant of the deformation gradient $F$, i.e. that there exists a function $\mathcal{F}:\R^+\to\R$ such that $\frac1n\tr\sigma = \mathcal{F}(\det F)$, where $\sigma$ denotes the Cauchy stress tensor; cf.\ \cite{richter1948isotrope,sansour2008physical,charrier1988existence}.}, i.e.\ energy functions $W$ of the form
\label{sectionContains:applicationToIsoVolSplit}
\[
	W:\GLpz\to\R\,,\quad W(F) = W^{\mathrm{iso}}(F) + W^{\mathrm{vol}}(\det F) \;=\; W^{\mathrm{iso}}\left(\frac{F}{(\det F)^\ahalf}\right) + W^{\mathrm{vol}}(\det F)
\]
with an isochoric function $W^{\mathrm{iso}}\colon\GLpz\to\R$ and a function $W^{\mathrm{vol}}:\R^+\to\R$.
In this case, Theorem \ref{theorem:mainResult} and Propositions \ref{prop:mainResultInTermsOfLogSquared} and \ref{prop:mainResultInTermsOfK} provide sufficient criteria for the polyconvexity of $W$: if $W^{\mathrm{vol}}$ is convex on $\R^+$, then the polyconvexity of $W^{\mathrm{iso}}$ is sufficient for $W$ to be polyconvex as well. For example, since the mapping $t\mapsto \frac{\kappa}{2\,\widehat{k}}\,e^{\widehat{k}\,[(\log t)]^2}$ is convex on $\R^+$ for $\widehat{k}\geq\frac18$, it follows from Corollary \ref{cor:henckyExamples} that the exponentiated Hencky energy $W_{_{\rm eH}}:\GLpz\to\R$ with
\[
	W_{_{\rm eH}}(F)=W_{_{\rm eH}}^{\mathrm{iso}}\left(F\right)+W_{_{\rm eH}}^{\mathrm{vol}}(\det F) = \dd\frac{\mu}{k}\,e^{k\,\|\dev_2\log U\|^2}+\frac{\kappa}{2\,\widehat{k}}\,e^{\widehat{k}\,[(\log \det U)]^2}
\]
is polyconvex for $k\geq\frac14$ and $\widehat{k}\geq\frac18$, as indicated in Section \ref{section:introduction}.

\subsection{Growth conditions for polyconvex isochoric energies}\setcounter{equation}{0}
By integrating the polyconvexity criteria given in Proposition \ref{prop:mainResultInTermsOfLogSquared}, we obtain an exponential growth condition for the function $f$ which is necessarily satisfied if $W$ is rank-one convex (i.e.\ polyconvex).
 \begin{corollary}
 Let $W:\GLpz\to\R$ with $W(F)=\ftilde(\|\dev_2\log U\|^2)=f\left(\log^2 \frac{\lambda_1}{\lambda_2}\right)$ be a polyconvex energy function with $f\in C^2([0,\infty))$. If $f^\prime(\theta) \neq 0$ for all $\theta>0$, then the function $f$ satisfies the inequality
 \begin{align}
f(\theta)\geq (e^{\sqrt{\theta}}-1)\,\frac{\sqrt{\varepsilon}}{e^{\sqrt{\varepsilon}}}\,f^{\prime}(\varepsilon)\,
+f(0) \qquad \text{for all }\;\theta,\varepsilon>0\,.
\end{align}
 \end{corollary}
 \begin{proof}
 According to Proposition \ref{prop:mainResultInTermsOfLogSquared} and Corollary \ref{cor:polyconvexityImpliesMonotonicityITOf}, if the energy $W$ is polyconvex, then
\begin{align}
2\,\theta\,  f^{\prime\prime}(\theta) + (1-\sqrt{\theta})\,f^{\prime}(\theta)\geq 0 \quad\text{ and }\quad f^{\prime}(\theta)\geq0 \quad\text{ for all }\;\theta> 0\,.
\end{align}
Under our assumption $f^{\prime}(\theta)\neq0$, we therefore find $f^{\prime}(\theta)>0$ for all $\theta>0$ and deduce
\begin{align}
\frac{  f^{\prime\prime}(\theta)}{ f^{\prime}(\theta)}\geq \frac{\sqrt{\theta}-1}{2\,\theta}  \quad\text{ for all }\;\theta>0\,.
\end{align}
By integration from $\varepsilon>0$ to $\theta$, it follows that
\begin{align}
\log  f^{\prime}(\theta)\geq \log  f^{\prime}(\varepsilon) +\frac{1}{2}\left(2\sqrt{\theta}-\log\theta\right)-\frac{1}{2}\left(2\sqrt{\varepsilon}-\log\varepsilon\right) \quad\text{ for all }\;\theta,\varepsilon>0\,,
\end{align}
thus we obtain
\begin{align}
 f^{\prime}(\theta)\geq e^{\log  f^{\prime}(\varepsilon) +\frac{1}{2}\left(2\sqrt{\theta}-\log\theta\right)-\frac{1}{2}\left(2\sqrt{\varepsilon}-\log\varepsilon\right)}=
 f^{\prime}(\varepsilon)\,e^{-\sqrt{\varepsilon}+\frac{1}{2}\log\varepsilon}\,e^{\sqrt{\theta}-\frac{1}{2}\log\theta}
 =
 f^{\prime}(\varepsilon)\,\frac{\sqrt{\varepsilon}}{e^{\sqrt{\varepsilon}}}\,\frac{e^{\sqrt{\theta}}}{\sqrt{\theta}}
\end{align}
for all $\theta,\varepsilon>0$. By another integration on the interval $[\delta,\theta]$, $\delta>0$, we obtain
\begin{align}
f(\theta)\geq f^{\prime}(\varepsilon)\,\frac{\sqrt{\varepsilon}}{e^{\sqrt{\varepsilon}}}\,e^{\sqrt{\theta}}
+f(\delta)-f^{\prime}(\varepsilon)\,\frac{\sqrt{\varepsilon}}{e^{\sqrt{\varepsilon}}}\,e^{\sqrt{\delta}}
\end{align}
for all $\theta,\varepsilon,\delta>0$. Taking the limit case $\delta\rightarrow 0$ and using the continuity of the function $f$, we finally obtain
\begin{align}
f(\theta)\geq f^{\prime}(\varepsilon)\,\frac{\sqrt{\varepsilon}}{e^{\sqrt{\varepsilon}}}\,e^{\sqrt{\theta}}
+f(0)-f^{\prime}(\varepsilon)\,\frac{\sqrt{\varepsilon}}{e^{\sqrt{\varepsilon}}} \quad\text{ for all }\;\theta,\varepsilon>0,
\end{align}
and the proof is complete.
\end{proof}
\begin{remark} Since $f^\prime(\theta)\geq 0$ for all $\theta\geq 0$, a necessary condition is that
\begin{equation}
\label{eq:growthConditionITOf}
	f(\theta) \;\geq\; C_1 \,e^{\sqrt{\theta}}+C_2 \quad\text{ for all }\;\theta>0\,,
\end{equation}
for $C_1=\frac{1}{e}\,f^{\prime}(1)>0$ and $C_2=f(0)-\frac{1}{e}\,f^{\prime}(1)\in\R$. In terms of the function $h$ with $W(F)=h\big(\frac{\lambda_1}{\lambda_2}\big)$, inequality \eqref{eq:growthConditionITOf} also implies
\[
	h(t) \;\geq\; C_1\,t + C_2 \quad\text{ for all }\; t>1\,,
\]
since $h(t) = f(\log^2 t)$.
\end{remark}
Sendova and Walton \cite{Walton05} derive similar necessary growth conditions for the three-dimensional case. Growth conditions for polyconvex functions have also been considered by Yan \cite{yan1997rank}, who showed that non-constant polyconvex conformal energy functions defined on all of $\R^{n\times n}$ must grow at least with power $n$.

\newpage

\bibliographystyle{plain} %
\addcontentsline{toc}{section}{References}

\begin{footnotesize}

\begin{thebibliography}{10}

\bibitem{alibert1992example}
J.-J. Alibert and B.~Dacorogna.
\newblock An example of a quasiconvex function that is not polyconvex in two
  dimensions.
\newblock {\em Archive for Rational Mechanics and Analysis}, 117(2):155--166,
  1992.

\bibitem{eremeyev2007constitutive}
H.~Altenbach, V.~Eremeyev, L.~Lebedev, and L.A. Rend{\'o}n.
\newblock Acceleration waves and ellipticity in thermoelastic micropolar media.
\newblock {\em Arch. Appl. Mech.}, 80(3):217--227, 2010.

\bibitem{astala2012quasiconformal}
K.~Astala, T.~Iwaniec, I.~Prause, and E.~Saksman.
\newblock Burkholder integrals, {M}orrey's problem and quasiconformal mappings.
\newblock {\em Journal of the American Mathematical Society}, 25(2):507--531,
  2012.

\bibitem{aubert1987counterexample}
G~Aubert.
\newblock On a counterexample of a rank 1 convex function which is not
  polyconvex in the case {$N=2$}.
\newblock {\em Proceedings of the Royal Society of Edinburgh: Section A
  Mathematics}, 106(3-4):237--240, 1987.

\bibitem{aubert1995necessary}
G.~Aubert.
\newblock Necessary and sufficient conditions for isotropic rank-one convex
  functions in dimension 2.
\newblock {\em Journal of Elasticity}, 39(1):31--46, 1995.

\bibitem{Ball78}
J.M. Ball.
\newblock Constitutive inequalities and existence theorems in nonlinear
  elastostatics.
\newblock In R.J. Knops, editor, {\em Herriot {W}att {S}ymposion: {N}onlinear
  {A}nalysis and {M}echanics.}, volume~1, pages 187--238. Pitman, London, 1977.

\bibitem{Ball77}
J.M. Ball.
\newblock Convexity conditions and existence theorems in nonlinear elasticity.
\newblock {\em Archive for Rational Mechanics and Analysis}, 63:337--403, 1977.

\bibitem{Balldoes}
J.M. Ball.
\newblock Does rank one convexity imply quasiconvexity?
\newblock {\em Preprint 262, Inst. for Math. and its Appl., Univ. of Minnesota
  at Minneapolis}, 1983.

\bibitem{ball1990sets}
J.M. Ball.
\newblock Sets of gradients with no rank-one connections.
\newblock {\em Journal de Math{\'e}matiques Pures et Appliqu{\'e}es},
  69(3):241--259, 1990.

\bibitem{Ball02}
J.M. Ball.
\newblock Some open problems in elasticity.
\newblock In P.~Newton et~al., editor, {\em Geometry, mechanics, and
  dynamics.}, pages 3--59. Springer, New-York, 2002.

\bibitem{ball2015incompatible}
J.M. Ball and R.D. James.
\newblock Incompatible sets of gradients and metastability.
\newblock {\em Archive for Rational Mechanics and Analysis}, pages 1--54, 2015.

\bibitem{Ball84b}
J.M. Ball and F.~Murat.
\newblock {$W^{1,p}$}-quasiconvexity and variational problems for multiple
  integrals.
\newblock {\em J. Funct. Anal.}, 58:225--253, 1984.

\bibitem{bertram2007rank}
A.~Bertram, T.~B{\"o}hlke, and M.~{\v{S}}ilhav{\`y}.
\newblock On the rank 1 convexity of stored energy functions of physically
  linear stress-strain relations.
\newblock {\em Journal of Elasticity}, 86(3):235--243, 2007.

\bibitem{Brighi97}
M.~Bousselsal and B.~Brighi.
\newblock Rank-one-convex and quasiconvex envelopes for functions depending on
  quadratic forms.
\newblock {\em J. Convex Anal.}, 4:305--319, 1997.

\bibitem{Bruhns01}
O.T. Bruhns, H.~Xiao, and A.~Mayers.
\newblock Constitutive inequalities for an isotropic elastic strain energy
  function based on {H}encky's logarithmic strain tensor.
\newblock {\em Proc. Roy. Soc. London A}, 457:2207--2226, 2001.

\bibitem{charrier1988existence}
P.~Charrier, B.~Dacorogna, B.~Hanouzet, and P.~Laborde.
\newblock An existence theorem for slightly compressible materials in nonlinear
  elasticity.
\newblock {\em SIAM Journal on Mathematical Analysis}, 19(1):70--85, 1988.

\bibitem{chaudhuri2003rank}
N.~Chaudhuri and S.~M{\"u}ller.
\newblock Rank-one convexity implies quasi-convexity on certain hypersurfaces.
\newblock {\em Proceedings of the Royal Society of Edinburgh: Section A
  Mathematics}, 133(06):1263--1272, 2003.

\bibitem{chelminski2006new}
K.~Che{\l}mi{\'n}ski and A.~Ka{\l}amajska.
\newblock New convexity conditions in the calculus of variations and
  compensated compactness theory.
\newblock {\em ESAIM: Control, Optimisation and Calculus of Variations},
  12(01):64--92, 2006.

\bibitem{chirictua1999time}
S.~Chiri{\c{t}}{\u{a}} and M.~Ciarletta.
\newblock Time-weighted surface power function method for the study of spatial
  behaviour in dynamics of continua.
\newblock {\em Eur. J. Mech.-A/Solids}, 18(5):915--933, 1999.

\bibitem{chirictua2007strong}
S.~Chiri{\c{t}}{\u{a}}, A.~Danescu, and M.~Ciarletta.
\newblock On the strong ellipticity of the anisotropic linearly elastic
  materials.
\newblock {\em Journal of Elasticity}, 87(1):1--27, 2007.

\bibitem{ChiritaGhiba1}
S.~Chiri{\c t}\u{a} and I.D. Ghiba.
\newblock Strong ellipticity and progressive waves in elastic materials with
  voids.
\newblock {\em Proc. R. Soc. A}, 466:439--458, 2010.

\bibitem{conti2008quasiconvex}
S.~Conti.
\newblock Quasiconvex functions incorporating volumetric constraints are
  rank-one convex.
\newblock {\em Journal de Math{\'e}matiques Pures et Appliqu{\'e}es},
  90(1):15--30, 2008.

\bibitem{conti2005rank}
S.~Conti, D.~Faraco, F.~Maggi, and S.~M{\"u}ller.
\newblock Rank-one convex functions on 2$\times$2 symmetric matrices and
  laminates on rank-three lines.
\newblock {\em Calc. Var. Partial Differential Equations}, 24(4):479--493,
  2005.

\bibitem{conti2003polyconvexity}
S.~Conti, C.~De Lellis, S.~M{\"u}ller, and M.~Romeo.
\newblock Polyconvexity equals rank-one convexity for connected isotropic sets
  in {$\mathbb{M}^{2\times2}$}.
\newblock {\em Compt. Rend. Math.}, 337(4):233--238, 2003.

\bibitem{criscione2000invariant}
J.C. Criscione, J.D. Humphrey, A.S. Douglas, and W.C. Hunter.
\newblock An invariant basis for natural strain which yields orthogonal stress
  response terms in isotropic hyperelasticity.
\newblock {\em Journal of the Mechanics and Physics of Solids},
  48(12):2445--2465, 2000.

\bibitem{Dacorogna01}
B.~Dacorogna.
\newblock Necessary and sufficient conditions for strong ellipticity of
  isotropic functions in any dimension.
\newblock {\em Discrete Cont. Dyn. Syst.}, 1(2):257--263, 2001.

\bibitem{Dacorogna08}
B.~Dacorogna.
\newblock {\em Direct {M}ethods in the {C}alculus of {V}ariations.}, volume~78
  of {\em Applied {M}athematical {S}ciences}.
\newblock Springer, Berlin, 2. edition, 2008.

\bibitem{davies1991simple}
P.J. Davies.
\newblock A simple derivation of necessary and sufficient conditions for the
  strong ellipticity of isotropic hyperelastic materials in plane strain.
\newblock {\em Journal of Elasticity}, 26(3):291--296, 1991.

\bibitem{de2012note}
D.~De~Tommasi, G.~Puglisi, and G.~Zurlo.
\newblock A note on strong ellipticity in two-dimensional isotropic elasticity.
\newblock {\em Journal of Elasticity}, 109(1):67--74, 2012.

\bibitem{dolzmann2012}
G.~Dolzmann.
\newblock Regularity of minimizers in nonlinear elasticity -- the case of a
  one-well problem in nonlinear elasticity.
\newblock {\em Technische Mechanik}, 32:189--194, 2012.

\bibitem{dolzmann2013}
G.~Dolzmann, J.~Kristensen, and K.~Zhang.
\newblock {BMO} and uniform estimates for multi-well problems.
\newblock {\em Manuscripta Mathematica}, 140(1-2):83--114, 2013.

\bibitem{edelstein1968note}
W.~Edelstein and R.~Fosdick.
\newblock A note on non-uniqueness in linear elasticity theory.
\newblock {\em Z. Angew. Math. Phys.}, 19(6):906--912, 1968.

\bibitem{ernst1998ellipticity}
E.~Ernst.
\newblock Ellipticity loss in isotropic elasticity.
\newblock {\em Journal of Elasticity}, 51(3):203--211, 1998.

\bibitem{fosdick2007note}
R.~Fosdick, M.D. Piccioni, and G.~Puglisi.
\newblock A note on uniqueness in linear elastostatics.
\newblock {\em Journal of Elasticity}, 88(1):79--86, 2007.

\bibitem{GhibaAn}
I.D. Ghiba.
\newblock On the spatial behaviour of harmonic vibrations in an elastic
  cylinder.
\newblock {\em An. St. Univ. Iasi, Sect. Matematica}, 52(f.1):75--86, 2006.

\bibitem{ghiba2015spatial}
I.D. Ghiba and E.~Bulgariu.
\newblock On spatial evolution of the solution of a non-standard problem in the
  bending theory of elastic plates.
\newblock {\em IMA J. Appl. Math.}, 80(2):452--473, 2015.

\bibitem{ghiba2015exponentiated}
I.D. Ghiba, P.~Neff, and M.~{\v{S}}ilhav{\'y}.
\newblock The exponentiated {Hencky}-logarithmic strain energy. {Improvement}
  of planar polyconvexity.
\newblock {\em Int. J. Non-Linear Mech.}, 71:48--51, 2015.

\bibitem{heinz2015quasiconvexity}
S.~Heinz.
\newblock Quasiconvexity equals lamination convexity for isotropic sets of
  {$2\times2$} matrices.
\newblock {\em Advances in Calculus of Variations}, 8(1):43--53, 2015.

\bibitem{Hutchinson82}
J.W. Hutchinson and K.W. Neale.
\newblock Finite strain {${J}_2$}-deformation theory.
\newblock In D.E. Carlson and R.T. Shield, editors, {\em Proceedings of the
  IUTAM Symposium on Finite Elasticity}, pages 237--247. Martinus Nijhoff,
  1982,
  \texttt{https://www.uni-due.de/imperia/md/content/mathematik/ag\_neff/hutchinson\_ellipticity80.pdf}.

\bibitem{iwaniec2009}
T.~Iwaniec and J.~Onninen.
\newblock Hyperelastic deformations of smallest total energy.
\newblock {\em Archive for Rational Mechanics and Analysis}, 194(3):927--986,
  2009.

\bibitem{iwaniec2011invitation}
T.~Iwaniec and J.~Onninen.
\newblock An invitation to n-harmonic hyperelasticity.
\newblock {\em Pure and Applied Mathematics Quarterly}, 7(2), 2011.

\bibitem{kalamajska2003new}
A.~Ka{\l}amajska.
\newblock On new geometric conditions for some weakly lower semicontinuous
  functionals with applications to the rank-one conjecture of morrey.
\newblock {\em Proceedings of the Royal Society of Edinburgh: Section A
  Mathematics}, 133(06):1361--1377, 2003.

\bibitem{kalamajska2015}
A.~Ka{\l}amajska and P.~Kozarzewski.
\newblock On the condition of tetrahedral polyconvexity arising from calculus
  of variations.
\newblock 2015.
\newblock Preprint available at
  \url{http://www.mimuw.edu.pl/badania/preprinty/preprinty-imat/papers/77/pr-imat-77.pdf}.

\bibitem{knowles1975ellipticity}
J.K. Knowles and E.~Sternberg.
\newblock On the ellipticity of the equations of nonlinear elastostatics for a
  special material.
\newblock {\em Journal of Elasticity}, 5(3-4):341--361, 1975.

\bibitem{knowles1976failure}
J.K. Knowles and E.~Sternberg.
\newblock On the failure of ellipticity of the equations for finite
  elastostatic plane strain.
\newblock {\em Archive for Rational Mechanics and Analysis}, 63(4):321--336,
  1976.

\bibitem{kreiner2006topology}
C.-F. Kreiner and J.~Zimmer.
\newblock Topology and geometry of nontrivial rank-one convex hulls for
  two-by-two matrices.
\newblock {\em ESAIM: Control, Optimisation and Calculus of Variations},
  12(2):253--270, 2006.

\bibitem{marcellini1984quasiconvex}
P.~Marcellini.
\newblock Quasiconvex quadratic forms in two dimensions.
\newblock {\em Applied Mathematics and Optimization}, 11(1):183--189, 1984.

\bibitem{morrey1952quasi}
C.B. Morrey.
\newblock Quasi-convexity and the lower semicontinuity of multiple integrals.
\newblock {\em Pacific J. Math}, 2(1):25--53, 1952.

\bibitem{muller1999rank}
S.~M\"uller.
\newblock Rank-one convexity implies quasiconvexity on diagonal matrices.
\newblock {\em Int. Math. Res. Not.}, 1999(20):1087--1095, 1999.

\bibitem{ndanou2014hyperbolicity}
S.~Ndanou, N.~Favrie, and S.~Gavrilyuk.
\newblock Criterion of {H}yperbolicity in {H}yperelasticity in the {C}ase of
  the {S}tored {E}nergy in {S}eparable {F}orm.
\newblock {\em Journal of Elasticity}, 115(1):1--25, 2014.

\bibitem{Neff_Diss00}
P.~Neff.
\newblock {\em Mathematische {A}nalyse multiplikativer {V}iskoplastizit\"at.
  {P}h.{D}. {t}hesis, {Technische Universit\"at Darmstadt}.}
\newblock Shaker Verlag, ISBN:3-8265-7560-1,
  \texttt{https://www.uni-due.de/$\sim$hm0014/Download\_files/cism\_convexity08.pdf},
  Aachen, 2000.

\bibitem{Neff_critique05}
P.~Neff.
\newblock Critique of ``{T}wo-dimensional examples of rank-one convex functions
  that are not quasiconvex" by {M.K. B}enaouda and {J.J. T}elega.
\newblock {\em Annales Polonici Mathematici}, 86.2:193--195, 2005.

\bibitem{Neff_Osterbrink_Martin_hencky13}
P.~Neff, B.~Eidel, and R.J. Martin.
\newblock Geometry of logarithmic strain measures in solid mechanics.
\newblock {\em Preprint arXiv:1505.02203}, 2015.

\bibitem{neff2013hencky}
P.~Neff, B.~Eidel, F.~Osterbrink, and R.~Martin.
\newblock The {H}encky strain energy $\|\log {U}\|^2$ measures the geodesic
  distance of the deformation gradient to $\rm{SO(3)}$ in the canonical
  left-invariant {R}iemannian metric on $\rm{GL(3)}$.
\newblock {\em PAMM}, 13(1):369--370, 2013.

\bibitem{NeffEidelOsterbrinkMartin_Riemannianapproach}
P.~{Neff}, B.~{Eidel}, F.~{Osterbrink}, and R.~{Martin}.
\newblock {A Riemannian approach to strain measures in nonlinear elasticity}.
\newblock {\em C. R. Acad. Sci.}, 342:254--257, 2014.

\bibitem{NeffGhibaPlasticity}
P.~Neff and I.D. Ghiba.
\newblock The exponentiated {Hencky}-logarithmic strain energy.{ Part III:
  Coupling} with idealized isotropic finite strain plasticity.
\newblock {\em to appear in Cont. Mech. Thermod., arXiv:1409.7555}, the special
  issue in honour of D.J. Steigmann, 2015.

\bibitem{NeffGhibaLankeit}
P.~Neff, I.D. Ghiba, and J.~Lankeit.
\newblock The exponentiated {H}encky-logarithmic strain energy. {P}art {I}:
  {C}onstitutive issues and rank--one convexity.
\newblock {\em to appear in Journal of Elasticity, doi:
  10.1007/s10659-015-9524-7}, arXiv:1403.3843, 2015.

\bibitem{NeffGhibaPoly}
P.~Neff, I.D. Ghiba, J.~Lankeit, R.~Martin, and D.J. Steigmann.
\newblock The exponentiated {H}encky-logarithmic strain energy. {P}art {II}:
  {C}oercivity, planar polyconvexity and existence of minimizers.
\newblock {\em to appear in Z. Angew. Math. Phys., doi:
  10.1007/s00033-015-0495-0, arXiv:1408.4430}, 2015.

\bibitem{Ogden83}
R.W. Ogden.
\newblock {\em Non-{L}inear {E}lastic {D}eformations.}
\newblock Mathematics and its Applications. Ellis Horwood, Chichester, 1.
  edition, 1983.

\bibitem{parry1995planar}
G.P. Parry.
\newblock On the planar rank-one convexity condition.
\newblock {\em Proceedings of the Royal Society of Edinburgh: Section A
  Mathematics}, 125(02):247--264, 1995.

\bibitem{parry2000rank}
G.P. Parry and M.~{\v{S}}ilhav{\`y}.
\newblock On rank one connectedness, for planar objective functions.
\newblock {\em Journal of Elasticity}, 58(2):177--189, 2000.

\bibitem{pedregal1996some}
P.~Pedregal.
\newblock Some remarks on quasiconvexity and rank-one convexity.
\newblock {\em Proceedings of the Royal Society of Edinburgh: Section A
  Mathematics}, 126(05):1055--1065, 1996.

\bibitem{pedregal2014some}
P.~Pedregal.
\newblock Some evidence in favor of {Morrey's} conjecture.
\newblock {\em Preprint arXiv:1406.7199}, 2014.

\bibitem{Sverak98}
P.~Pedregal and V.~{\v{S}}ver{\'a}k.
\newblock A note on quasiconvexity and rank-one convexity for $2\times2$
  matrices.
\newblock {\em J. Convex. Anal.}, 5:107--117, 1998.

\bibitem{Raoult86}
A.~Raoult.
\newblock Non-polyconvexity of the stored energy function of a
  {S}t.{V}enant-{K}irchhoff material.
\newblock {\em Aplikace Matematiky}, 6:417--419, 1986.

\bibitem{richter1948isotrope}
H.~Richter.
\newblock {Das isotrope Elastizit{\"a}tsgesetz}.
\newblock {\em Z. Angew. Math. Mech.}, 28(7-8):205--209, 1948,\\
  \texttt{https://www.uni-due.de/imperia/md/content/mathematik/ag\_neff/richter\_isotrop\_log.pdf}.

\bibitem{sansour2008physical}
C.~Sansour.
\newblock On the physical assumptions underlying the volumetric-isochoric split
  and the case of anisotropy.
\newblock {\em Eur. J. Mech.-A/Solids}, 27(1):28--39, 2008.

\bibitem{SawyersRivlin78}
K.N. {Sawyers} and R.~{Rivlin}.
\newblock {On the speed of propagation of waves in a deformed compressible
  elastic material.}
\newblock {\em {Z. Angew. Math. Phys.}}, 29:245--251, 1978.

\bibitem{Walton05}
T.~{Sendova} and J.~R. {Walton}.
\newblock {On strong ellipticity for isotropic hyperelastic materials based
  upon logarithmic strain.}
\newblock {\em {Int. J. Non-Linear Mech.}}, 40(2-3):195--212, 2005.

\bibitem{serre1983}
D.~Serre.
\newblock Formes quadratiques et calcul des variations.
\newblock {\em Journal de Math{\'em}atiques Pures et Appliqu{\'e}es},
  62:177--196, 1983.

\bibitem{SilhavyPRE99}
M.~{\v{S}}ilhav{\'y}.
\newblock On isotropic rank one convex functions.
\newblock {\em Proceedings of the Royal Society of Edinburgh}, 129:1081--1105,
  1999.

\bibitem{silhavy2002convexity}
M.~{\v{S}}ilhav{\'y}.
\newblock Convexity conditions for rotationally invariant functions in two
  dimensions.
\newblock In A.~Sequeira, H.B. da~Veiga, and J.H. Videman, editors, {\em
  Applied Nonlinear Analysis}, pages 513--530. Springer US, 2002.

\bibitem{simpson2008bifurcation}
H.~Simpson and S.~Spector.
\newblock On bifurcation in finite elasticity: buckling of a rectangular rod.
\newblock {\em Journal of Elasticity}, 92(3):277--326, 2008.

\bibitem{sivaloganathan1988implications}
J.~Sivaloganathan.
\newblock Implications of rank one convexity.
\newblock In {\em Annales de l'Institut Henri Poincar{\'e} Analyse non
  lin{\'e}aire}, volume~5, pages 99--118, 1988.

\bibitem{Sverak92}
V.~{\v{S}}ver{\'a}k.
\newblock Rank-one convexity does not imply quasiconvexity.
\newblock {\em Proceedings of the Royal Society of Edinburgh: Section A
  Mathematics}, 120:185--189, 1992.

\bibitem{terpstra1939darstellung}
F.J. Terpstra.
\newblock Die {D}arstellung biquadratischer {F}ormen als {S}ummen von
  {Q}uadraten mit {A}nwendung auf die {V}ariationsrechnung.
\newblock {\em Mathematische Annalen}, 116(1):166--180, 1939.

\bibitem{volberg2012ahlfors}
A.~Volberg.
\newblock {Ahlfors-Beurling} operator on radial functions.
\newblock {\em Preprint arXiv:1203.2291}, 2012.

\bibitem{yan1997rank}
B.~Yan.
\newblock On rank-one convex and polyconvex conformal energy functions with
  slow growth.
\newblock {\em Proceedings of the Royal Society of Edinburgh: Section A
  Mathematics}, 127(3):651--663, 1997.

\bibitem{ZubovRudev}
L.M. {Zubov} and A.N. {Rudev}.
\newblock {A criterion for the strong ellipticity of the equilibrium equations
  of an isotropic nonlinearly elastic material.}
\newblock {\em {J. Appl. Math. Mech.}}, 75:432--446, 2011.

\end{thebibliography}

%

\end{footnotesize}

\newpage
\appendix
\section{Appendix}
\addtocontents{toc}{\protect\setcounter{tocdepth}{0}}
\setcounter{section}{1}
\subsection{Additional examples and applications}
\label{appendix:examples}
The criteria given in Sections \ref{section:mainResult} and \ref{section:criteriaDifferentRepresentations} can be applied to a number of isochoric energy functions in order to determine whether or not they are polyconvex or, equivalently, rank-one convex.
\begin{corollary}~\\
The following functions $W:\GLpz\to\R$ are rank-one convex and polyconvex:
\begin{itemize}
\item[i)] $\displaystyle W(F) = \bigg\|\frac{U}{(\det U)^{\ahalf}} - \left(\frac{U}{(\det U)^{\ahalf}}\right)^{-1}\bigg\|^2$\,,
\item[ii)] $\displaystyle W(F) = e^{\norm{\dev_2\log U}^2}\cdot\frac{\norm{F}^2}{\det F}$\,,
\item[iii)] $W(F) = \cosh(\norm{\dev\log U}^2) = \cosh(\norm{\dev\log \sqrt{F^TF}}^2)$\,,
\end{itemize}
The following functions $W:\GLpz\to\R$ are neither rank-one convex nor polyconvex:
\begin{itemize}
\item[iv)] $\displaystyle W(F) = \norm{\dev_2\log U}^\beta$ for $\beta>0$\,,
\item[v)] $\displaystyle W(F) = e^{\norm{\dev_2\log U}^2+\sin(\norm{\dev_2\log U}^2)}$\,.
\end{itemize}
\end{corollary}
\begin{proof}
\begin{itemize}
\item[i)] Since the squared Frobenius-norm of a symmetric matrix $X$ is the squared sum of its eigenvalues, we find for $F\in\GLpz$ with singular values $\lambda_1,\lambda_2$:
\[
	W(F) = \left(\frac{\lambda_1}{\sqrt{\lambda_1\,\lambda_2}} - \frac{\lambda_1^{-1}}{\sqrt{\lambda_1^{-1}\,\lambda_2^{-1}}}\right)^2
		+ \left(\frac{\lambda_2}{\sqrt{\lambda_1\,\lambda_2}} - \frac{\lambda_2^{-1}}{\sqrt{\lambda_1^{-1}\,\lambda_2^{-1}}}\right)^2
	= 2\,\left(\sqrt{\frac{\lambda_1}{\lambda_2}} - \sqrt{\frac{\lambda_2}{\lambda_1}}\right)^2
\]
Thus the function $h:\R^+\to\R^+$ with $h(t)=h(\frac1t)$ with $W(F)=h\big(\frac{\lambda_1}{\lambda_2}\big)$ for all $F\in\GLpz$ with singular values $\lambda_1,\lambda_2$ is given by
\[
	h(t) = 2\,\left(\sqrt{t}-\frac{1}{\sqrt{t}}\right)^2 = 2\,\left(t + \frac1t\right) - 4\,,
\]
and we find
\[
	h^{\prime}(t) = 2\,\left(1-\frac{1}{t^2}\right) \qquad\text{as well as}\qquad h^{\prime\prime}(t) = \frac{4}{t^3} \geq 0
\]
for all $t\in\R^+$. Thus, according to Theorem \ref{theorem:mainResult}, $W$ is polyconvex.
\item[ii)] Again, we write $W(F)$ in terms of the singular values $\lambda_1,\lambda_2$ of $F$:
\[
	W(F) \;=\; e^{\norm{\dev_2\log U}^2}\cdot\frac{\norm{F}^2}{\det F}
	\;=\; e^{\frac12\log^2\frac{\lambda_1}{\lambda_2}}\cdot\frac{\lambda_1^2 + \lambda_2^2}{\lambda_1\,\lambda_2}
	\;=\; e^{\frac12\log^2\frac{\lambda_1}{\lambda_2}}\cdot\left(\frac{\lambda_1}{\lambda_2} + \frac{\lambda_2}{\lambda_1}\right)\,.
\]
Then $W(F)=h\big(\frac{\lambda_1}{\lambda_2}\big)$, where
\[
	h(t) = e^{\frac12\log^2 t}\cdot\left(t+\frac1t\right)\,,
\]
and we compute
\[
	h^{\prime}(t) = e^{\frac12\log^2 t}\cdot\left[1-\frac{1}{t^2}+\frac{\log t}{t}\,\left(t+\frac1t\right)\right] = e^{\frac12\log^2 t}\cdot\left[1-\frac{1}{t^2}+{\log t}\cdot\left(1+\frac{1}{t^2}\right)\right]
\]
as well as
\begin{align*}
	h^{\prime\prime}(t) &= e^{\frac12\log^2 t}\cdot\left[ \frac{\log t}{t}\,\left(1-\frac{1}{t^2} + \log t\cdot\left(1+\frac{1}{t^2}\right)\right)
		+ \left(\frac{2}{t^3}+\frac1t\,\left(1+\frac{1}{t^2}\right)-\frac{2\log t}{t^3}\right) \right]\\
	&= e^{\frac12\log^2 t}\cdot\left[ \frac{\log t}{t} - \frac{\log t}{t^3} + \left(\frac1t+\frac{1}{t^3}\right)\,\log^2 t + \frac{2}{t^3} + \frac1t + \frac{1}{t^3} - \frac{2\log t}{t^3} \right]\\
	&= e^{\frac12\log^2 t}\cdot\left[ \frac1t + \frac{3}{t^3} + \left(\frac1t-\frac{3}{t^3}\right)\,\log t + \left(\frac1t+\frac{1}{t^3}\right)\,\log^2(t) \right]\,.
\end{align*}
Therefore, Theorem \ref{theorem:mainResult} states that $W$ is polyconvex if and only if
\begin{equation}
\label{eq:exampleInequality}
	\left(\frac{3}{t^2}-1\right)\,\log t \;\leq\; 1 + \frac{3}{t^2} + \left(1+\frac{1}{t^2}\right)\,\log^2 t
\end{equation}
for all $t>0$. For $t<1$ or $t<\sqrt{3}$, the left-hand side is negative and the inequality is therefore satisfied. If $1\leq t \leq \sqrt{3}$, then $0\leq \log t <1$ and $\frac{3}{t^2}-1\geq0$; thus
\[
	\left(\frac{3}{t^2}-1\right)\,\log t \;\leq\; \frac{3}{t^2}-1 \;<\; 1 + \frac{3}{t^2} + \left(1+\frac{1}{t^2}\right)\,\log^2 t\,,
\]
hence inequality \eqref{eq:exampleInequality} is satisfied in this case as well.
\item[iii)] The function $\ftilde:[0,\infty)\to\R$ with $W(F) = \ftilde(\norm{\dev_2\log U}^2)$ for all $F\in\GLpz$ is given by $\ftilde(\eta)=\cosh(\eta)$. For $\eta\in\R^+$ we find
\[
	2\,\eta\,\ftilde^{\prime\prime}(\eta)+ (1-\sqrt{2\,\eta})\,\ftilde^{\prime}(\eta) \;=\; 2\,\eta\,\cosh(\eta)+ (1-\sqrt{2\,\eta})\,\sinh(\eta) \;\geq\; (2\,\eta + 1-\sqrt{2\,\eta})\,\sinh(\eta) \;\geq\; 0\,,
\]
thus $W$ is polyconvex according to Proposition \ref{prop:mainResultInTermsOfLogSquared}.
\item[iv)] Let $\alpha\colonequals\frac\beta2$. Then $W(F)=\ftilde(\norm{\dev_2\log U}^2)$ for $\ftilde(\eta)=\eta^\alpha$. Since
\[
	2\,\eta\,\ftilde^{\prime\prime}(\eta) + (1-\sqrt{2\,\eta})\,\ftilde^{\prime}(\eta) = 2\,\eta\,\alpha\,(\alpha-1)\,\eta^{\alpha-2} + (1-\sqrt{2\eta})\,\alpha\,\eta^{\alpha-1} = \alpha\,\eta^{\alpha-1}\,\left[2\,\alpha-1-\sqrt{2\,\eta}\right]\,,
\]
we use Proposition \ref{prop:mainResultInTermsOfLogSquared} to find that $W$ is polyconvex if and only if
\[
	0 \;\leq\; 2\,\alpha-1-\sqrt{2\,\eta} \quad\text{ for all }\; \eta\in\R^+\,,
\]
which is obviously not the case for any $\beta=2\,\alpha>0$. This result was also hinted at by Hutchinson and Neale \cite{Hutchinson82}.
\item[v)] We apply Proposition \ref{prop:mainResultInTermsOfLogSquared} to the function $\ftilde$ with $\ftilde(\eta)=e^{\eta+\sin \eta}$. Since
\[
	\ftilde^{\prime}(\eta) = e^{\eta+\sin \eta}\cdot (1+\cos \eta) \quad\text{ and }\quad \ftilde^{\prime\prime}(\eta) = e^{\eta+\sin \eta}\cdot \big( (1+\cos \eta)^2 - \sin\eta \big)\,,
\]
we find
\[
	2\,\eta\,\ftilde^{\prime\prime}(\eta) + (1-\sqrt{2\,\eta})\,\ftilde^{\prime}(\eta) = 2\,\eta\,e^{\eta+\sin \eta}\cdot \big( (1+\cos \eta)^2 - \sin\eta \big) + (1-\sqrt{2\,\eta})\,e^{\eta+\sin \eta}\cdot (1+\cos \eta)
\]
Thus $W$ is polyconvex if and only if
\[
	2\,\eta\,\big( (1+\cos \eta)^2 - \sin\eta \big) + (1-\sqrt{2\,\eta})\,(1+\cos \eta) \;\geq\; 0 \quad\text{ for all }\; \eta\in(0,\infty)\,.
\]
This inequality is not satisfied for $\eta=\frac\pi2$. Note that $\ftilde$ is monotone on $\R^+$ with exponential growth, but is not convex.
\end{itemize}
\end{proof}

\subsection{On $\dist^2\left(\frac{F}{(\det F)^{\ahalf}}\,,\;\SO(2)\right)$}\setcounter{equation}{0}
\label{appendix:distances}
For $F\in\GLpz$, we consider the squared distance from $\frac{F}{(\det F)^{1\!/\!2}}\in\SL(2)$ to the special orthogonal group $\SO(2)$ with respect to different distance measures. Such distances are closely connected to a number of elastic energy functions, including the isochoric quadratic Hencky energy \cite{Neff_Osterbrink_Martin_hencky13}, and they provide an important class of examples for isochoric energy functions on $\GLpz$. In this appendix, we collect some related results which are scattered throughout the literature.
\subsubsection{The Euclidean distance of $F\in\R^{2\times 2}$ to $\SO(2)$}
We first consider the Euclidean distance
\[
	\disteuc^2\left({F},\SO(2)\right)\colonequals\inf_{R\in \SO(2)}\|F-R\|^2
\]
of $F\in\mathbb{R}^{2\times 2}$ to $\SO(2)$, where $\norm{\,.\,}$ denotes the Frobenius matrix norm. In the two-dimensional case, this distance can be explicitly calculated: since
\begin{align*}
\disteuc^2\left({F},\SO(2)\right)=\inf_{R\in \SO(2)}\|F-R\|^2=\inf_{\alpha\in [-\pi, \pi]}\Big\|F-\left(
                                                                                                              \begin{array}{cc}
                                                                                                                \cos \alpha & \sin \alpha \\
                                                                                                                -\sin \alpha & \cos \alpha \\
                                                                                                              \end{array}
                                                                                                            \right)
\Big\|^2\,,
\end{align*}
we find
\begin{align*}
\Big\|\left(
 \begin{array}{cc}
  F_{11}-\cos \alpha & F_{12}-\sin \alpha \\
   F_{21}+\sin \alpha & F_{22}-\cos \alpha \\
    \end{array}
     \right)
\Big\|^2=(F_{11}-\cos\alpha)^2+(F_{12}-\sin \alpha)^2+(F_{21}+\sin \alpha)^2+(F_{22}-\cos \alpha)^2\,.
\end{align*}
Taking the derivative with respect to $\alpha$ yields the stationarity condition
\begin{align*}
(F_{11}+F_{22})\,\sin \alpha+(F_{21}-F_{12})\,\cos \alpha=0\qquad \Longleftrightarrow\qquad \langle \left(
                                                                                                  \begin{array}{c}
                                                                                                    \sin \alpha \\
                                                                                                    \cos \alpha \\
                                                                                                  \end{array}
                                                                                                \right),
                                                                                                \left(
                                                                                                  \begin{array}{c}
                                                                                                    F_{11}+F_{22} \\
                                                                                                    F_{21}-F_{12}\\
                                                                                                  \end{array}
                                                                                                \right)\rangle =0\,,
\end{align*}
which implies
\begin{align*}
\left(
                                                                                                  \begin{array}{c}
                                                                                                    \sin \alpha \\
                                                                                                    \cos \alpha \\
                                                                                                  \end{array}
                                                                                                \right)
=\pm \frac{1}{\sqrt{\|F\|^2+2\, \det F}}\, \left(
                                                                                                  \begin{array}{c}
                                                                                                    -(F_{21}-F_{12}) \\
                                                                                                    F_{11}+F_{22}\\
                                                                                                  \end{array}
                                                                                                \right)\,.
\end{align*}
The minimum is easily seen to be realized by
\begin{align*}
\left(
\begin{array}{c}
\sin \alpha \\
 \cos \alpha \\
 \end{array}
 \right)
=\frac{1}{\sqrt{\|F\|^2+2\, \det F}}\, \left(
\begin{array}{c}
-(F_{21}-F_{12}) \\
 F_{11}+F_{22}\\
 \end{array}
 \right)\,,
\end{align*}
and reinserting yields
\begin{align*}
\disteuc^2\left({F},\SO(2)\right) = \inf_{R\in \SO(2)}\| F-R\|^2=\|F\|^2-2\, \sqrt{\|F\|^2+2\, \det F}+2
\end{align*}
for arbitrary $F\in \mathbb{R}^{2\times 2}$.
Let us recall the Biot energy term
\begin{align*}
W_{\rm Biot}(F)=\|U-\id\|^2=\|U\|^2-2\, \tr(U)+2.
\end{align*}
For $F\in\GLpz$, the Caley-Hamilton formula implies that
\begin{align*}
\|U\|^2-[\tr(U)]^2+2\, \det U=0\qquad \Longrightarrow\qquad \tr(U)=\sqrt{\|U\|^2+2\, \det U}\overset{F\in \GLpz}{=}\sqrt{\|F\|^2+2\, \det F}\,,
\end{align*}
hence
\begin{align*}
W_{\rm Biot}(F)&=\|F\|^2-2\, \sqrt{\|F\|^2+2\, \det F}+2=\left(\sqrt{\|F\|^2+2\, \det F}-1\right)^2+1-2\, \det F\\
&\hspace{-0.5cm}\overset{F\in \GLpz}{=}\|U\|^2-2\, \sqrt{\|U\|^2+2\, \det U}+2
\end{align*}
and we note that
\begin{equation}
\label{eq:wBiotResult}
	\disteuc^2(F,\SO(2)) = \norm{U-\id}^2 = W_{\rm Biot}(F) \quad\text{ for all }\;F\in\GLpz\,,
\end{equation}
while in general $\disteuc^2(F,\SO(2))\geq W_{\rm Biot}(F)$ for $F\in\R^{2\times2}$. Note that $W_{\rm Biot}$ is not rank-one convex \cite{bertram2007rank}.

\subsubsection{The polyconvexity of $F\mapsto\disteuc^2\left(\frac{F}{(\det F)^{\ahalf}}\,,\; \SO(2)\right)$}
\label{section:polyconvexityOfIsochoricDistance}
In order to show that the mapping $F\mapsto \disteuc^2\big(\frac{F}{(\det F)^{\ahalf}},\SO(2)\big)$ is polyconvex on $\GLpz$, we apply \eqref{eq:wBiotResult} to $\frac{F}{(\det F)^{\ahalf}}$ and find
\begin{align*}
	\disteuc^2\left(\frac{F}{(\det F)^{\ahalf}},\SO(2)\right) &= \left(\sqrt{\Big\|\frac{F}{(\det F)^{\ahalf}}\Big\|^2+2\, \det \Big(\frac{F}{(\det F)^{\ahalf}}}\Big)-1\right)^2+1-2\, \det \Big(\frac{F}{(\det F)^{\ahalf}}\Big)\notag\\
	&= \left(\sqrt{\frac{\|F\|^2}{\det F}+2}-1\right)^2-1.
\end{align*}
Since the function
\begin{align*}
t\mapsto \left(\sqrt{t+2}-1\right)^2-1.
\end{align*}
is convex and monotone, we only need to prove that the mapping $F\mapsto\frac{\norm{F}^2}{\det F}$ is polyconvex. This is shown (in a slightly generalized version) in the following lemma, using the criteria developed in Section \ref{section:criteriaDifferentRepresentations}.
\begin{lemma}
Let $\beta>0$. Then the function
\[
	W:\GLpz\to\R\,,\quad W(F) = \left(\frac{\norm{F}^2}{\det F}\right){\rule{0em}{1.2em}}^\beta
\]
is polyconvex (and, equivalently, rank-one convex) if and only if $\beta\geq1$.
\end{lemma}
\begin{proof}
The unique function $\z:[1,\to\infty)\to\R$ with $W(F) = \z\left(\frac12\,\frac{\norm{F}^2}{\det F}\right)$ for all $F\in\GLpz$ is given by $\z(r)=2^\beta\,r^\beta$. Then
\[
	\z^{\prime}(r) = 2^\beta\,\beta\,r^{\beta-1} \quad\text{ and }\quad \z^{\prime\prime}(r) = 2^\beta\,\beta\,(\beta-1)\,r^{\beta-2}\,,
\]
thus according to Proposition \ref{prop:mainResultInTermsOfK}, the function $W$ is polyconvex if and only if
\begin{align*}
	0 \;&\leq\; (r^2-1)\,(r+\sqrt{r^2-1})\,\z^{\prime\prime}(r) + z^{\prime}(r)\nonumber\\
	&=\; 2^\beta\,\beta\,r^{\beta-2}\,\left[(\beta-1)\,(r^2-1)\,(r+\sqrt{r^2-1}) + r\right] \quad\text{ for all }\; r>1\,. \label{eq:inequalityForPowerOfK}
\end{align*}
Since $2^\beta\,\beta\,r^{\beta-2}>0$ for all $\beta>0$ and $r>1$, this inequality is equivalent to
\begin{align*}
	0 \;&\leq\; (\beta-1)\,(r^2-1)\,(r+\sqrt{r^2-1}) + r\\
	\Longleftrightarrow\qquad \beta-1 \;&\geq\; -\,\frac{r}{(r^2-1)\,(r+\sqrt{r^2-1})} \qquad\qquad\text{ for all }\; r>1\,.
\end{align*}
The right hand side in the last equality is always negative, so the polyconvexity condition is satisfied for all $\beta\geq1$. Furthermore, the right hand expression tends to $0$ as $r$ tends to $\infty$, hence the condition cannot be satisfied for $\beta<1$.
\end{proof}
\noindent Note that, in the three-dimensional case, the mapping $F\mapsto\Big(\frac{\norm{F}^3}{\det F}\Big)^\beta$ is polyconvex if and only if $\beta\geq\frac12$, as shown in \cite[Proposition 6]{charrier1988existence}.

\subsubsection{The quasiconvex hull of $\disteuc^2\left({F},\, \SO(2)\right)$}

In contrast to the isochoric function $F\mapsto \disteuc^2\big(\frac{F}{(\det F)^{\ahalf}},\SO(2)\big)$, the squared Euclidean distance of $F$ to $\SO(2)$ is not polyconvex and not even rank-one convex. However, the quasiconvex hull of the function can be computed explicitly using the
Brighi--Theorem, adapted to the two-dimensional case:
\begin{theorem}{\rm \cite[Theorem 3.2, page 310]{Brighi97}}
\label{brighi-theorem}
Let $q:\mathbb{R}^{2\times 2} \rightarrow\mathbb{R}_+$ be a non-negative quadratic form. For a function $\varphi:\R^+\to[0,\infty)$, define $W:\mathbb{R}^{2\times 2}\rightarrow \mathbb{R}$ by
\begin{align*}
	W(F) = \varphi(q(F))\,.
\end{align*}
Let $\mu^*,\alpha\in\R$ be such that
\[
	\mu^*=\inf _{t\in \mathbb{R}_+}\varphi(t)=\varphi(\alpha)\,.
\]
Then
\begin{align*}
	R[W(F)]=Q[W(F)]=P[W(F)]=C[W(F)]=\mu^* \quad\text{ for all }\; F\in \mathbb{R}^{2\times 2} \;\text{ with }\; q(F)\leq \alpha\,,
\end{align*}
where $R[W(F)]$, $Q[W(F)]$, $P[W(F)]$ and $C[W(F)]$ denote the rank-one convex hull, the quasiconvex hull, the polyconvex hull and the convex hull of $W$, respectively.
\end{theorem}
\noindent We apply this theorem to $q:\mathbb{R}^{2\times 2} \rightarrow\mathbb{R}_+$ with
\begin{align*}
q(F)=\|F\|^2+2\, \det F\,.
\end{align*}
Note that $q$ is indeed a non-negative quadratic form due to the arithmetic-geometric mean inequality. Consider the function $\varphi:\mathbb{R}_+\rightarrow\mathbb{R}_+$ with
\begin{align*}
\varphi(t)=(\sqrt{t}-1)^2\quad \Rightarrow\quad \inf_{t\in \mathbb{R}_+}=0=\varphi(1) \quad \Rightarrow\quad  \mu^*=0\,, \ \ \alpha=1\,,
\end{align*}
and let
\[
	W(F) = \varphi(q(F)) = (\sqrt{\norm{F}^2+2\det F}-1)^2\,.
\]
From Theorem \ref{brighi-theorem} we conclude that
\begin{align*}
R[ W(F)]=Q[W(F)]=P[W(F)]=C[W(F)]=0 \quad\text{ for all }\;F\in\R^{2\times2} \;\text{ with }\;q(F)\leq 1\,.
\end{align*}
Now set
\begin{align*}
	\widehat{W}(F) \colonequals
	\left\{\begin{array}{cl}
		0 &:\; q(F)\leq 1  \vspace{1.5mm}\\
		(\sqrt{q(F)}-1)^2 &:\; q(F)\geq 1
	\end{array}\right.
	\;=\;
	\left\{\begin{array}{cc}
		0 &:\; \|F\|^2+2\, \det F\leq 1   \vspace{1.5mm}\\
		(\sqrt{\|F\|^2+2\, \det F}-1)^2 &:\; \|F\|^2+2\, \det F\geq 1
	\end{array}\right.\,.
\end{align*}
Then $\widehat{W}$ is convex (and therefore quasiconvex) as the composition $\widehat{W}=\widehat{\varphi}\circ q$ of the (convex) quadratic form $q$ and the non-decreasing convex function $\widehat{\varphi}:\R^+\to\R$ with
\[
	\widehat{\varphi}(t) \colonequals \left\{
	\begin{array}{cl}
	0 &:\; t\leq 1  \vspace{1.5mm}\\
	(\sqrt{t}-1)^2 &:\; t\geq 1
	\end{array}
	\right.\,.
\]
We observe that $\widehat{W}(F) = 0 = Q[W(F)]$ for all $F\in\R^{2\times2}$ with $q(F)\leq1$ and that $\widehat{W}(F)=W(F)\geq Q[W(F)]$ for all $F\in\R^{2\times2}$ with $q(F)>1$. Thus $\widehat{W}$ is a quasiconvex function with $\widehat{W}(F) \geq Q[W(F)]$ and $\widehat{W}(F) \leq W(F)$ for all $F\in\R^{2\times2}$, hence $\widehat{W}$ is the quasiconvex hull of $W$:
\[
	Q[W(F)]
	= \widehat{W}(F) =
	\left\{\begin{array}{cc}
		0 &:\; \|F\|^2+2\, \det F\leq 1   \vspace{1.5mm}\\
		(\sqrt{\|F\|^2+2\, \det F}-1)^2 &:\; \|F\|^2+2\, \det F\geq 1
	\end{array}\right.\,.
\]
Taking the representation
\begin{align*}
\disteuc^2(F,{\rm SO}(2))=\left(\sqrt{\|F\|^2+2\, \det F}-1\right)^2 + 1-2\,\det F = W(F)  + 1-2\,\det F\,,
\end{align*}
it is easy to see that
\begin{align*}
Q[\disteuc^2(F,{\rm SO}(2))] %
= Q[W(F)] +1-2\,\det F = \widehat{W}(F) +1-2\,\det F\,,
\end{align*}
since $F\mapsto 1-2\, \det F$ is a Null-Lagrangian. We therefore find
\begin{align*}
	Q[\disteuc^2(F,\SO(2))]&=\left\{
	\begin{array}{cc}
		1-2\, \det F &:\; \|F\|^2+2\, \det F\leq 1  \vspace{1.5mm} \\
		(\sqrt{\|F\|^2+2\, \det F}-1)^2+1-2\, \det F &:\; \|F\|^2+2\, \det F\geq 1
	\end{array}\right.\\
	&=\left\{
	\begin{array}{cc}
		1-2\, \det F  &:\; \|F\|^2+2\, \det F\leq 1  \vspace{1.5mm} \\
		\hspace{4.2em}\disteuc^2(F,\,\SO(2)) \hspace{3.98em} &:\; \|F\|^2+2\, \det F\geq 1
	\end{array}\right.
\end{align*}
for $F\in\R^{2\times2}$. The same result has been given by Dolzmann \cite{dolzmann2012,dolzmann2013} with an alternative proof. The quasiconvex hull of the mapping $F\mapsto \disteuc^2(F,\SO(3))$ is not yet known.

\subsubsection{A comparison of distance functions on $\GLpz$}
Let $\dg(F,\SO(2))=\norm{\log U}^2$ denote the \emph{geodesic distance} \cite{neff2013hencky,Neff_Osterbrink_Martin_hencky13,NeffEidelOsterbrinkMartin_Riemannianapproach} of $F$ to $\SO(2)$. Then we can list the following convexity properties of (modified) distance functions to $\SO(2)$:\\[2em]
\newcommand{\tabvspace}{\vphantom{\rule{0em}{1.4em}}}%
\begin{tabular}{ll}
\textbullet\tabvspace\ $\disteuc^2\left({F},\SO(2)\right)=\|U-\id\|^2$& \textbf{is not}  rank-one convex \cite{bertram2007rank},\\[1em]
\textbullet\tabvspace\ $\disteuc^2\left(\frac{F}{(\det F)^{\ahalf}},\SO(2)\right) = \left\|\frac{U}{(\det U)^\ahalf}-\id\right\|^2$& \textbf{is} polyconvex (Section \ref{section:polyconvexityOfIsochoricDistance}),\\[1em] %
\textbullet\tabvspace\ $\dg^2\left({F},\SO(2)\right)=\|\log U\|^2$& \textbf{is not} rank-one convex \cite{Bruhns01,Neff_Diss00},\\[1em]
\textbullet\tabvspace\ $\dg^2\left(\frac{F}{(\det F)^{\ahalf}},\SO(2)\right)=\|\dev_2\log U\|^2\quad$& \textbf{is not} rank-one convex \cite{Neff_Diss00},\\[1em]
\textbullet\tabvspace\ $e^{\dg^2\left({F},\SO(2)\right)}=e^{\|\log U\|^2}$& \textbf{is not} rank-one convex \cite{NeffGhibaLankeit},\\[1em]
\textbullet\tabvspace\ $e^{\dg^2\left(\frac{F}{(\det F)^{\ahalf}},\SO(2)\right)}=e^{\|\dev_2\log U\|^2}$& \textbf{is} polyconvex \cite{ghiba2015exponentiated}.
\end{tabular}
\end{document}